\documentclass[11pt]{amsart}

\usepackage{geometry}                
\usepackage{graphicx}
\usepackage{tikz}
\usepackage{amssymb}
\usepackage{epstopdf}
\usepackage{amsmath}
\usepackage{color}
\usepackage{url}
\usepackage{stmaryrd}
\usepackage{mathtools}
\usepackage{array}
\usepackage{multirow}

\theoremstyle{definition}

\newtheorem*{ack}{Acknowledgements}
\newtheorem*{notat}{Notation}
\newtheorem*{thm*}{Theorem}
\newtheorem*{rem}{Remark}
\theoremstyle{plain}

\newtheorem{thm}{Theorem}[section]
\newtheorem{lem}[thm]{Lemma}
\newtheorem{prop}[thm]{Proposition}
\newtheorem{corol}[thm]{Corollary}

\numberwithin{equation}{section}
\makeatletter
\newcommand*{\@old@slash}{}\let\@old@slash\slash
\def\slash{\relax\ifmmode\delimiter"502F30E\mathopen{}\else\@old@slash\fi}
\makeatother

\def\backslash{\delimiter"526E30F\mathopen{}}

\DeclareMathOperator{\num}{num}

\DeclareMathOperator{\Gal}{Gal}

\DeclareMathOperator{\GL}{GL}
\DeclareMathOperator{\SL}{SL}

\newcommand{\proj}{\mathbb{P}}

\newcommand{\Q}{\mathbb{Q}}
\newcommand{\Z}{\mathbb{Z}}

\newcommand{\C}{\mathbb{C}}

\newcommand{\F}{\mathbb{F}}

\DeclareMathOperator{\Aut}{Aut}

\newcommand{\uC}{\mathcal{H}}

\newcommand{\idp}{\frak{p}}

\DeclareMathOperator{\Frob}{Frob}

\DeclareMathOperator{\End}{End}

\DeclareMathOperator{\PGL}{PGL}

\newcommand{\loc}{{\mathcal O}}

\DeclareMathOperator{\Hom}{Hom}

\DeclareMathOperator{\Res}{Res}

\DeclareMathOperator{\Tr}{Tr}

\DeclareMathOperator{\Spec}{Spec}

\DeclareMathOperator{\Tors}{Tors}

\DeclareMathOperator{\Cot}{Cot}

\title{Some cases of Serre's uniformity problem}
\author{Pedro Lemos}
\address{Max Planck Institute for Mathematics\\Vivatsgasse 7\\Bonn 53111\\Germany}
\email{lemos.pj@gmail.com}
\date{\today}
\begin{document}
\maketitle

\begin{abstract}
We show that if $E/\Q$ is an elliptic curve without complex multiplication and for which there is a prime $q$ such that the image of $\bar{\rho}_{E,q}$ is contained in the normaliser of a split Cartan subgroup of $\GL_2(\F_q)$, then $\bar{\rho}_{E,p}$ surjects onto $\GL_2(\F_p)$ for every prime $p>37$. This result complements a previous result by the author. We also prove analogue results for certain families of $\Q$-curves, building on results of Ellenberg (2004) and Le~Fourn (2016). 
\end{abstract}
\section{Introduction}
Let $K$ be a number field and $E$ an elliptic curve defined over $K$. Given a prime number~$p$, we will denote the mod $p$ Galois representation obtained from the Galois action on the $p$-torsion points of $E(\bar{K})$ (where $\bar{K}$ is an algebraic closure of $K$) by $\bar{\rho}_{E,p}$. The image of this representation is contained in $\GL(E[p])$, which is (non-canonically) isomorphic to $\GL_2(\F_p)$. We will often implicitly make a choice of an $\F_p$-basis for $E[p]$ and regard $\bar{\rho}_{E,p}$ as having image contained in $\GL_2(\F_p)$. Throughout this paper, we will say that $\bar{\rho}_{E,p}$ is \emph{surjective} if its image is the whole of $\GL_2(\F_p)$. The question of determining under what conditions these representations are surjective is very important in modern number theory. One of the earliest and most striking results in this area is due to Serre.
\begin{thm}[{\cite[Th\'eor\`eme~2]{ser_prop}}]\label{serthm}
Let $K$ be a number field and let $E$ be an elliptic curve defined over $K$ and without complex multiplication. There exists a constant $C_{E,K}$ such that $\bar{\rho}_{E,p}$ is surjective for every prime $p>C_{E,K}$.
\end{thm}

Serre's uniformity problem (see section 4.3 of~\cite{ser_prop}) asks to what extent the constant $C_{E,K}$ of the theorem above is dependent on $E$. More precisely, it asks whether there exists a constant $C_K$ depending only on $K$ such that, given an elliptic curve $E$ defined over $K$ and without complex multiplication, the residual mod $p$ Galois representation $\bar{\rho}_{E,p}$ is surjective for every prime $p>C_K$. An affirmative answer to this question would be likely to yield important applications in the study of certain Diophantine equations, as the work of Darmon and Merel~\cite{darmer} shows.

The most studied and most well understood case is, naturally, the one where $K=\Q$. The strongest result we have to this date is the following.
\begin{thm}[\cite{bilpar,bilparreb,maz_eis,maz_rat,ser_prop}]\label{soa}
Let $E$ be an elliptic curve defined over $\Q$  and without complex multiplication. Let $p$ be a prime number strictly larger than $37$. If $\bar{\rho}_{E,p}$ is not surjective, then its image is contained in the normaliser of a non-split Cartan subgroup of $\GL_2(\F_p)$.
\end{thm}

In~\cite{lem}, the author showed that the normaliser of a non-split Cartan case cannot occur for primes $p>37$ if an elliptic curve as in the theorem above admits a non-trivial cyclic isogeny defined over $\Q$. 
\begin{thm}[{\cite[Theorem~1.1]{lem}}]\label{lem1}
Let $E$ be an elliptic curve defined over $\Q$ and without complex multiplication. Suppose that $E$ admits a non-trivial cyclic isogeny defined over $\Q$. Then $\bar{\rho}_{E,p}$ is surjective for every prime $p>37$.
\end{thm}

Another way of saying that an elliptic curve defined over a number field $K$ admits a non-trivial cyclic isogeny defined over $K$ is by saying that there exists a prime $q$ for which the image of $\bar{\rho}_{E,q}:G_K\rightarrow \GL_2(\F_q)$ is contained in a Borel subgroup of $\GL_2(\F_q)$. It is then natural to ask whether we can obtain results of the same kind if we replace ``Borel subgroup'' by another maximal subgroup of $\GL_2(\F_q)$. In the first part of this paper, we show that the same result holds if this maximal subgroup is chosen to be the normaliser of a split Cartan. More precisely, we show the following theorem.

\begin{thm}\label{qmaintheorem}
Let $E/\Q$ be an elliptic curve without complex multiplication. Suppose that there exists a prime $q$ for which the image of $\bar{\rho}_{E,q}$ is contained in the normaliser of a split Cartan subgroup of $\GL_2(\F_q)$. Then $\bar{\rho}_{E,p}$ is surjective for every $p>37$.
\end{thm}
Note that it follows from the work of Bilu, Parent and Rebolledo~\cite{bilparreb} that there are only finitely many primes $q$ for which there exists a non-CM elliptic curve defined over $\Q$ such that the image of $\bar{\rho}_{E,q}$ is contained in the normaliser of a split Cartan subgroup of $\GL_2(\F_q)$. More precisely, they show that $q\in\{2,3,5,7,13\}$. Moreover, by the recent work of Balakrishnan, Dogra, M\"uller, Tuitman and Vonk~\cite{balak}, the prime $13$ is not on this list, and the list is reduced to $\{2,3,5,7\}$.

In order to prove this theorem, we follow the same strategy employed to prove Theorem~\ref{lem1}, namely, we start by showing that if $E$ is an elliptic curve satisfying the conditions of Theorem~\ref{qmaintheorem} and such that the image of $\bar{\rho}_{E,p}$ is contained in the normaliser of a non-split Cartan subgroup for some prime $p\geq 11$, then its $j$-invariant is integral.
\begin{prop}\label{int_j}
Let $E/\Q$ be an elliptic curve without complex multiplication. Suppose that there exists a prime $p\geq 11$ for which the image of the residual Galois representation $\bar{\rho}_{E,p}$ is contained in the normaliser of a non-split Cartan subgroup of $\GL_2(\F_p)$. Suppose, moreover, that there exists a prime $q$ different from $p$ such that the image of $\bar{\rho}_{E,q}$ is contained in the normaliser of a split Cartan subgroup of $\GL_2(\F_q)$. Then the $j$-invariant of $E$ is integral.
\end{prop}
This result is proven an adaptation of Mazur's formal immersion argument (see~\cite{maz_eis,maz_rat}).

By Theorem~\ref{soa}, the only elliptic curves which could consitute a contradiction to Theorem~\ref{qmaintheorem} are those for which there exists a prime $p>37$ such that the image of $\bar{\rho}_{E,p}$ is contained in the normaliser of a non-split Cartan, and so they must all have integral $j$-invariants. Using explicit parametrisations of the $j$-invariant maps for $X_0(q)$, where $q$ is an element of the set $\{2,3,5,7\}$, we find out that there are only finitely many $\Q$-points of $X_0(q)$ with integral $j$-invariant. Moreover, we are able to compute all the possible $j$-invariants. As any two elliptic curves with the same $j$-invariant are related to each other by a quadratic twist as long as their $j$-invariant is not $0$ nor $1728$, surjectivity only depends on the $j$-invariant, and so our problem is reduced to computing the largest non-surjective prime for a finite set of elliptic curves.

The second part of this paper is devoted to $\Q$-curves. Let us just recall a few definitions before proceeding. Let $E$ be an elliptic curve defined over a Galois number field $K$. Given an element $\sigma\in\Gal(K/\Q)$, we will denote by ${}^{\sigma}E$ the Galois conjugate of $E$ by $\sigma$. Recall that $E$ is said to be a $\Q$-curve if, for each $\sigma\in\Gal(K/\Q)$, there exists an isogeny $\mu_{\sigma}:{}^{\sigma}E\rightarrow E$. If $E/K$ is a $\Q$-curve, we shall say that it is \emph{completely defined over $K$} if all of the isogenies $\mu_{\sigma}$ can be chosen in such a way that they are all defined over $K$. The main results of this paper make reference to some representations attached to $\Q$-curves that, following the notation introduced by Ellenberg~\cite{ellen_qcurv,ellen}, we will denote by $\proj\bar{\rho}_{E,p}$. Despite the notation, these are \emph{not}, in general, simply the projectivisations of $\bar{\rho}_{E,p}$ (the projectivisation of $\bar{\rho}_{E,p}$ is, by definition, the composition of $\bar{\rho}_{E,p}$ with the canonical projection $\GL_2(\F_p)\rightarrow\PGL_2(\F_p)$); in fact, $\proj\bar{\rho}_{E,p}$ is defined on the whole of $G_{\Q}$, and not only on $G_K$, where $K$ is the number field over which $E$ is defined. However, there is a close relation between $\proj\bar{\rho}_{E,p}$ and $\bar{\rho}_{E,p}$: if $P\bar{\rho}_{E,p}$ stands for the projectivisation of $\bar{\rho}_{E,p}$, then $P\bar{\rho}_{E,p}$ is isomorphic to $\proj\bar{\rho}_{E,p}|_{G_K}$. For a brief review of the definition of $\proj\bar{\rho}_{E,p}$, we refer the reader to section~\ref{review}. When $K$ is a quadratic field, we say that a $\Q$-curve completely defined over $K$ is of \emph{degree $d$} if there exists an isogeny $\mu_{\sigma}:{}^{\sigma} E\rightarrow E$ defined over $K$ and of degree $d$ and there exists no other isogeny between ${}^{\sigma}E$ and $E$ of smaller degree, where $\sigma\in\Gal(K/\Q)$ is the non-trivial element.

The main objective of the second part of the paper is to prove the following results (which are analogues of Theorem~\ref{lem1} and Theorem~\ref{qmaintheorem}).

\begin{thm}\label{borelmain}
Let $K$ be a quadratic field and let $d$ be a square-free integer. There exists a constant $C_{K,d}$ satisfying the following property. If $E$ is a $\Q$-curve completely defined over $K$, of degree $d$, without complex multiplication and for which there exists a prime $q\nmid d$ such that the image of $\proj\bar{\rho}_{E,q}$ is contained in a Borel subgroup of $\PGL_2(\F_q)$, then $\proj\bar{\rho}_{E,p}$ surjects onto $\PGL_2(\F_p)$ for every $p>C_{K,d}$. 
\end{thm}

\begin{thm}\label{cartanmain}
Let $K$ be a quadratic field and let $d\notin\{2,3,5,7,13\}$ be a square-free integer. There exists a constant $C_{K,d}$ satisfying the following property. If $E$ is a $\Q$-curve completely defined over $K$, of degree $d$, without complex multiplication and for which there exists a prime $q\nmid d$ such that the image of $\proj\bar{\rho}_{E,q}$ is contained in the normaliser of a split Cartan subgroup of $\PGL_2(\F_q)$, then $\proj\bar{\rho}_{E,p}$ surjects onto $\PGL_2(\F_p)$ for every $p>C_{K,d}$.
\end{thm}
Most of the proof of these two theorems will use arguments of the same type of those used to prove Theorem~\ref{qmaintheorem} and described above. In particular, borrowing some ideas of Ellenberg~\cite{ellen}, we will show the following.
\begin{prop}\label{borel_main}
Let $K$ be a quadratic number field and let $d$ be a square-free positive integer. Let $E$ be a $\Q$-curve completely defined over $K$, of degree $d$ and without complex multiplication. Suppose that $p$ and $q$ are distinct primes not dividing $d$ such that the image of $\proj\bar{\rho}_{E,p}$ is contained in the normaliser of a non-split Cartan subgroup of $\PGL_2(\F_p)$ and that the image of $\proj\bar{\rho}_{E,q}$ is contained in a Borel subgroup of $\PGL_2(\F_q)$. Suppose, moreover, that $p\geq 11$. Then the $j$-invariant of $E$ is in $\loc_K$, where $\loc_K$ stands for the ring of integers of $K$.
\end{prop}
We remark that if $q\geq 11$ and $q\neq 13,17,41$, a much stronger result has been proven by Le Fourn~{\cite[Proposition~3.3]{lefourn}}. For the proof of Theorem~\ref{cartanmain}, we will actually use the following result from~\cite{lefourn}.
\begin{prop}[{\cite[Proposition~3.6]{lefourn}}]\label{cartan_main}
Let $K$ be a quadratic field and let $p=11$ or $p>13$ be a prime. Suppose that $E$ is a $\Q$-curve of square-free degree $d$ coprime to $p$ such that the image of $\proj\bar{\rho}_{E,p}$ is contained in the normaliser of a split Cartan subgroup. Then $j(E)\in\loc_K$.
\end{prop}

Finally, we would like to mention a theorem that will be used as an auxiliary result in the proof of Theorem~\ref{borelmain}, but which is interesting in its own right.
\begin{thm}\label{noborel}
Let $K$ be a quadratic number field and $d$ a positive square-free integer. There exists a constant $C_{K,d}$ satisfying the following property. Let $E/K$ be a $\Q$-curve completely defined over $K$, of degree $d$ and without complex multiplication. If $p\nmid d$ is a prime for which the image of $\proj\bar{\rho}_{E,p}$ is contained in a Borel subgroup of $\PGL_2(\F_p)$, then $p\leq C_{K,d}$. Moreover, if we restrict ourselves to the case where $p\equiv 1\pmod{4}$, then the constant $C_{K,d}$ can be chosen to be
\begin{equation*}
2^{6fc+1}(2^{6fc}+1),
\end{equation*}
where $c$ is the narrow class number of $K$ and $f$ is the residual degree of a prime of $K$ lying above~$2$ (which is independent of the prime above $2$ chosen). In particular, when $p\equiv 1\pmod{4}$, the constant $C_{K,d}$ is actually independent from $d$.
\end{thm}
The reader is referred to the paper of Le Fourn~\cite{lefourn}, where results of a similar nature are proven. Specifically, in~{\cite[Corollary~5.1]{lefourn}}, Le Fourn gives a bound for such primes that depends not only on the quadratic number field $K$, but also on the elliptic curve itself. However, by restricting himself to the cases where $K$ is imaginary quadratic, he is able to give the absolute bound of $2\cdot 10^{13}$ for the size of such primes (this is~{\cite[Theorem~5.4]{lefourn}}). In comparison, Theorem~\ref{noborel} shows the existence of a bound depending only on the quadratic number field $K$ and on $d$, regardless of whether $K$ is real or imaginary.

As a final remark, we would like, once again, to draw the reader's attention to the papers of Ellenberg~\cite{ellen} and Le Fourn~\cite{lefourn}. In~\cite{ellen}, Ellenberg shows that if $K$ is an imaginary quadratic field and $d\geq 2$ is a square-free integer, then there exists a constant $C_{K,d}$ such that, given a $\Q$-curve $E$ completely defined over $K$, of degree $d$ and without complex multiplication, either $\proj\bar{\rho}_{E,p}$ surjects onto $\PGL_2(\F_p)$ for every prime $p>C_{K,d}$, or $E$ has potentially godd reduction at every prime of $K$ of characteristic not dividing $6$. The arguments appearing in the $\Q$-curve section of this paper will be based on some of his ideas. In~\cite{lefourn}, Le Fourn improves on the results of Ellenberg and gives an upper bound depending only on the discriminant of $K$ (still assumed to be imaginary quadratic) and on the degree of the $\Q$-curve for the largest non-surjective prime associated to $E$. One peculiarity of their results is that they need the degree of the $\Q$-curve to be $\geq 2$, i.e., they do not prove anything for elliptic curves defined over $\Q$. In this paper, we will start by proving Theorem~\ref{qmaintheorem}, which is the analogue of Theorem~\ref{cartanmain} for elliptic curves defined over $\Q$, i.e., $\Q$-curves of degree $1$.

\begin{ack}
I want to express my gratitute to Filip Najman, Marusia Rebolledo and Samir Siksek for their time, patience and their valuable suggestions, which greatly helped me during the process of writing up this article. I also want to thank the referee for his corrections. Finally, I am also indebted to the Max Planck Institute for Mathematics, in Bonn, both for the financial support and for the excellent working environment.
\end{ack}

\section{Galois representations of $\Q$-curves}\label{review}
We follow the approach of Ellenberg~\cite{ellen_qcurv}. For a more conceptual and complete treatment of the material in this section, the reader is referred to~\cite{ribet_qcurv}. However, the description given here will suffice for the most part of the present article. Results from~\cite{ribet_qcurv} will only be used in the proof of Theorem~\ref{noborel}.

Let $K$ be a Galois number field. Let $E$ be a $\Q$-curve defined (but not necessarily completely defined) over $K$. Assume, moreover, that $E$ does not have complex multiplication. For each $\sigma\in\Gal(\bar{\Q}/\Q)$, choose an isogeny $\mu_{\sigma}:{}^{\sigma}E\rightarrow E$. Note that if the restriction of $\sigma$ to $K$ is the trivial automorphism, then ${}^{\sigma}E=E$, and, in this case, we can choose $\mu_{\sigma}$ to be the identity. We will always assume that we make this choice and that, moreover, if two elements $\sigma,\tau\in \Gal(\bar{\Q}/\Q)$ restrict to the same automorphism of $K$, then $\mu_{\sigma}=\mu_{\tau}$. Since $E$ does not have complex multiplication, we have $\End_{\bar{\Q}}(E)\otimes\Q=\Q$. Therefore, given $\sigma,\tau\in\Gal(\bar{\Q}/\Q)$, the element
\begin{equation*}
c_E(\sigma,\tau):=\frac{1}{\deg\mu_{\sigma\tau}}\mu_{\sigma}\circ{}^{\sigma}\mu_{\tau}\circ\hat{\mu}_{\sigma\tau}\in\End_{\bar{\Q}}(E)\otimes\Q,
\end{equation*}
where $\hat{\mu}_{\sigma\tau}$ stands for the dual isogeny of $\mu_{\sigma\tau}$, can be regarded as an element of $\Q^{\times}$. 

Given, a prime number $p$, let $T_p(E)$ be the $p$-adic Tate module of $E$. Define the function (which, in general, is \emph{not} a homomorphism) $\varpi_{E,p}:G_{\Q}\rightarrow\GL(T_p(E))\cong\GL_2(\Q_p)$ in the following manner: given $P\in T_p(E)$ and $\sigma\in G_{\Q}$, we impose that $\varpi_{E,p}(\sigma)(P)=\mu_{\sigma}({}^{\sigma}P)$. 
\begin{rem}
Note that ${}^{\sigma}P\in{}^{\sigma}E(\bar{K})$. So, if $\sigma$ does not restrict to the trivial automorphism of $K$, we may have ${}^{\sigma}P\notin E(\bar{K})$. 
\end{rem}It is straightforward to check that the action of
\begin{equation*}
\varpi_{E,p}(\sigma)\varpi_{E,p}(\tau)\varpi_{E,p}(\sigma\tau)^{-1}
\end{equation*}
on $T_p(E)$ is given by $c_E(\sigma,\tau)\in\Q^{\times}$. Thus, $\varpi_{E,p}$ gives rise to a well-defined homomorphism $\proj{\rho}_{E,p}: G_{\Q}\rightarrow\PGL_2(\Q_p)$. If $p$ does not divide the degree of any $\mu_{\sigma}$, the construction of $\proj\bar{\rho}_{E,p}$ is identical to this.

\section{The case of elliptic curves over $\Q$}\label{rational}

The aim of this section is to prove Proposition~\ref{int_j} and Theorem~\ref{qmaintheorem}. 

But before starting to prove the aforementioned results, let us introduce some notation and terminology that will be used throughout the paper. Table~\ref{table:2} contains a summary of facts and notation that we will need.

Recall that a subgroup $\Gamma$ of $\SL_2(\Z)$ is called a \emph{congruence subgroup} if there exists a positive integer $N$ such that it contains
\begin{equation*}
\Gamma(N):=\left\{\begin{pmatrix}a & b \\ c & d\end{pmatrix}\in\SL_2(\Z):a\equiv d\equiv 1\text{ and }b\equiv c\equiv 0\pmod{N}\right\}.
\end{equation*}
In Table~\ref{table:1} we list some of the congruence subgroups that will appear more frequently during the course of this paper. In this table, $N$ stands for a positive integer, $p$ for an odd prime number, and $r_p$ for the natural reduction map $\SL_2(\Z)\rightarrow \SL_2(\F_p)$. Moreover, given an odd prime number $p$, we fix a non-split Cartan subgroup $C_{\rm{ns}}(p)$ of $\GL_2(\F_p)$ and write $C_{\rm{ns}}^+(p)$ for its normaliser.
\begin{table}[h!]
\centering

\begin{tabular}{ |c|c| }
\hline
$\Gamma_0(N)$ & $\begin{pmatrix} a & b \\ c & d\end{pmatrix}\in\SL_2(\Z): c\equiv 0\pmod{N}$\\
\hline
$\Gamma_{\rm{sp}}(N)$ & $\begin{pmatrix} a & b \\ c & d\end{pmatrix}\in\SL_2(\Z): b\equiv c\equiv 0\pmod{N}$\\
\hline
$\Gamma_{\rm{sp}}^+(N)$ & $\begin{pmatrix} a & b \\ c & d\end{pmatrix}\in \SL_2(\Z):b\equiv c\equiv 0\text{ or }a\equiv d\equiv 0\pmod{N}$\\
\hline
$\Gamma_{\rm{ns}}(p)$ & $r_p^{-1}(C_{\rm{ns}}(p)\cap\SL_2(\F_p))$\\
\hline
$\Gamma_{\rm{ns}}^+(p)$ & $r_p^{-1}(C_{\rm{ns}}^+(p)\cap\SL_2(\F_p))$\\
\hline
\end{tabular}
\caption{Some congruence subgroups of $\SL_2(\Z)$.}
\label{table:1}
\end{table}

We will work with modular curves obtained as quotients of the extended upper half plane $\uC^*$ by one of the congruence subgroups above or by some intersections of them. In fact, for any congruence subgroup $\Gamma$ that we will work with, it can be shown that the Riemann surface $\Gamma\backslash\uC^*$ descends to an algebraic curve defined over $\Q$. The point on this curve corresponding to $i\infty$ will be known as the \emph{cusp at infinity} and will be denoted by~$\infty$. In the following table we set up some terminology and summarise some of the facts concerning to these modular curves that will reveal to be useful later.

\begin{table}[h!]
\centering
\begin{tabular}{|m{2cm}|m{2cm}|m{3cm}|m{3cm}|m{3cm}|}
\hline
Congruence subgroup & Modular curve & Degree of $j$-invariant map & Ramification of $\infty$ w.r.t. $j$ & Field of definition of $\infty$\\
\hline\hline
$\Gamma_0(N)$ & $X_0(N)$ & $N\prod_{p\mid N}(1+1/p)$ & $1$ & $\Q$\\
\hline
$\Gamma_{\rm{sp}}(p)$ & $X_{\rm{sp}}(p)$ & $p(p+1)$ & $p$ & $\Q$\\
\hline
$\Gamma_{\rm{sp}}^+(p)$ & $X_{\rm{sp}}^+(p)$ & $p(p+1)/2$ & $p$ & $\Q$\\
\hline
$\Gamma_{\rm{ns}}(p)$ & $X_{\rm{ns}}(p)$ & $p(p-1)$& $p$ & $\Q(\zeta_p)$\\
\hline
$\Gamma_{\rm{ns}}^+(p)$ & $X_{\rm{ns}}^+(p)$ &$p(p-1)/2$& $p$ & $\Q(\zeta_p+\zeta_p^{-1})$\\
\hline
\end{tabular}
\caption{Some modular curves.}
\label{table:2}
\end{table}
With the notation set up, we are now ready to prove the results we will need. We start by noting that if $E$ is an elliptic curve defined over $\Q$ and without complex multiplication, then the values of $q$ for which the image of $\bar{\rho}_{E,q}$ is contained in the normaliser of a split Cartan subgroup of $\GL_2(\F_q)$ are very restricted. In fact, we have the following result.
\begin{thm}[Bilu--Parent--Rebolledo~\cite{bilparreb}]\label{bpr}
Let $E/\Q$ be an elliptic curve without complex multiplication. Then, if $q$ is a prime such that $q=11$ or $q\geq 17$, the image of $\bar{\rho}_{E,q}$ cannot be contained in the normaliser of a split Cartan subgroup of $\GL_2(\F_q)$.
\end{thm}

Recently, Balakrishnan, Dogra, M\"uller, Tuitman and Vonk~\cite{balak} showed that the only $\Q$-rational points of $X_{\mathrm{sp}}^+(13)$ are its cusps, thus proving the following theorem.
\begin{thm}[{\cite[Theorem~1.1]{balak}}]\label{balak}
Let $E/\Q$ be an elliptic curve without complex multiplication. Then the image of $\bar{\rho}_{E,13}$ is not contained in the normaliser of a split Cartan subgroup of $\GL_2(\F_{13})$.
\end{thm}

Therefore, we are reduced to considering the cases where $q\in\{2,3,5,7\}$, i.e., the cases where the genus of $X_{\rm{sp}}^+(q)$ is $0$. However, further ahead, we will need some of the results in this section to hold in the case $q=13$ as well. In fact, Theorem~\ref{balak} will only be used in the proof of Theorem~\ref{qmaintheorem} in order to obtain the explicit bound of $37$ (see Theorem~\ref{qmaintheorem} below); up until then, we will always assume that $q\in\{2,3,5,7,13\}$. We remark that these are precisely the primes $q$ for which $X_0(q)$ has genus $0$, a fact that plays an important role in the proof of Proposition~\ref{standardformal}.

The following is a more general version of {\cite[Proposition~2.2]{lem}}. We will need this general form later.
\begin{prop}[cf. {\cite[Proposition~2.2]{lem}}]\label{goodp}
Let $K$ be a number field of degree $n$ and let $E$ be an elliptic curve defined over $K$. Let $p$ be a prime such that the image of $\bar{\rho}_{E,p}$ is contained in the normaliser of a non-split Cartan subgroup of $\GL_2(\F_p)$. If $E$ has potentially multiplicative reduction at a prime $\lambda$ not dividing $p$, then $N_{K/\Q}(\lambda)^2\equiv 1\pmod{p}$. Moreover, if $p>2n+1$, then $E$ has potentially good reduction at every prime of $K$ dividing~$p$. 
\end{prop}
\begin{proof}
Given a prime $\lambda$ of $K$, write $K_{\lambda}$ for the completion of $K$ at $\lambda$. Let $\bar{K}$ and $\bar{K}_{\lambda}$ be algebraic closures of $K$ and $K_{\lambda}$, respectively. Fix an embedding $\bar{K}\hookrightarrow\bar{K}_{\lambda}$. This induces an embedding of absolute Galois groups $G_{K_{\lambda}}\hookrightarrow G_K$, which amounts to a choice of a decomposition subgroup of $G_K$ over $\lambda$.

Now, suppose that $E$ has potentially multiplicative reduction at $\lambda$. Then we know that $E_{/K_{\lambda}}$ is a twist of a Tate curve $E_q$, $q\in K_{\lambda}^{\times}$.

Let $\psi$ be the character associated to this twist. It is well-known that $\psi$ is either trivial or quadratic. Therefore, we have
\begin{equation*}
\bar{\rho}_{E,p}|_{G_{K_{\lambda}}}\sim\begin{pmatrix} \psi\chi_p & * \\ 0 & \psi\end{pmatrix},
\end{equation*}
where $\chi_p:G_{K_{\lambda}}\rightarrow\F_p^{\times}$ stands for the mod $p$ cyclotomic character.
As a Cartan subgroup of $\GL_2(\F_p)$ is an index $2$ subgroup of its normaliser, $\bar{\rho}_{E,p}(\sigma)^2$ is an element of a non-split Cartan subgroup of $\GL_2(\F_p)$ for every $\sigma\in G_{K_{\idp}}$. Moreover, since $\psi$ is at most quadratic, the eigenvalues of $\bar{\rho}_{E,p}(\sigma)^2$ are $\chi_p(\sigma)^2$ and $1$. However, the eigenvalues of an element of a non-split Cartan subgroup are $\F_p$-conjugate. This means that $\chi_p(\sigma)^2=1$ for every $\sigma\in G_{K_{\lambda}}$. 

If $\lambda$ does not divide $p$, then this means that $N_{K/\Q}(\lambda)^2\equiv 1\pmod{p}$, as the statement of the proposition predicts.

Suppose now that $p>2n+1$ and that $\lambda$ divides $p$. Then, as $\chi_p(\sigma)^2=1$ for every $\sigma\in G_{K_{\lambda}}$, we must have $[K_{\lambda}(\zeta_p):K_{\lambda}]\leq 2$. On the other hand, $\Q_p(\zeta_p)\subseteq K_{\lambda}(\zeta_p)$ and $[\Q_p(\zeta_p):\Q_p]=p-1$. Hence, $[K_{\lambda}(\zeta_p):\Q_p]\geq p-1$, yielding $n\geq[K_{\lambda}:\Q_p]\geq(p-1)/2$, which contradicts the condition $p>2n+1$.
\end{proof}
In order to simplify notation, we will write $X_{\rm{sp,ns}}^{-,+}(q,p)$ for the curve $X_{\rm{sp}}(q)\times_{X(1)}X_{\rm{ns}}^+(p)$, $X_{\rm{sp,ns}}^{+,+}(q,p)$ for the curve $X_{\rm{sp}}^+(q)\times_{X(1)} X_{\rm{ns}}^+(p)$, and $X_{0,\rm{ns}}^+(N,p)$ for the curve $X_0(N)\times_{X(1)}X_{\rm{ns}}^+(p)$, where $N$ is a positive integer. These three curves correspond to certain quotients of the extended upper half plane: there is an analytic isomorphism between $X_{\rm{sp,ns}}^{-,+}(q,p)(\C)$ and the quotient of $\uC^*$ by $\Gamma_{\rm{sp}}(q)\cap\Gamma_{\rm{ns}}^+(p)$, another one between $X_{\rm{sp,ns}}^{+,+}(q,p)(\C)$ and the quotient of $\uC^*$ by $\Gamma_{\rm{sp}}^+(q)\cap\Gamma_{\rm{ns}}^+(p)$, and another between $X_{0,\rm{ns}}^+(N,p)(\C)$ and the quotient of $\uC^*$ by $\Gamma_0(N)\cap\Gamma_{\rm{ns}}^+(p)$.

In what follows, we will write $w_{q^2}$ for the involution of $X_{0,\mathrm{ns}}^+(q^2,p)$ arising from the Atkin--Lehner involution of $X_0(q^2)$ (recall that the moduli intepretation of the Atkin--Lehner involution of $X_0(q^2)$ is as follows: a point  of $X_0(q^2)$ represented by $(E,\varphi)$ --- where $E$ is an elliptic curve and $\varphi:E\rightarrow E'$ is an isogeny of degree $q^2$ --- is mapped to $(E',\hat{\varphi})$, where $\hat{\varphi}$ stands for the dual isogeny of $\varphi$).
\begin{lem}\label{isom_mod}
There is a $\Q$-isomorphism $\theta:X_{\rm{sp,ns}}^{-,+}(q,p)\rightarrow X_{0,\rm{ns}}^+(q^2,p)$. Moreover, the involution $w_{q^2}$ of $X_{0,\rm{ns}}^+(q^2,p)$ coming from the Atkin--Lehner involution of $X_0(q^2)$ corresponds, under this isomorphism, to the involution $\omega_q$ of $X_{\rm{sp,ns}}^{-,+}(q,p)$ coming from the obvious involution of $X_{\rm{sp}}(q)$. In other words, we have $\theta\circ\omega_q=w_{q^2}\circ \theta$.
\end{lem}
\begin{rem}
Even though there exists an isomorphism between $X_0(q^2)$ and $X_{\rm{sp}}(q)$, this is not enough to conclude Lemma~\ref{isom_mod}, because this isomorphism does not preserve $j$-invariants.
\end{rem}
\begin{proof}
Even though the existence of an isomorphism between $X_{\rm{sp,ns}}^{-,+}(q,p)$ and $X_{0,\rm{ns}}^+(q^2,p)$ cannot be directly proven by appealing to the isomorphism between $X_0(q^2)$ and $X_{\rm{sp}}(q)$, the proofs of the existence of these two isomorphisms are essentially the same. Indeed, start by identifying $X_{0,\rm{ns}}^{+}(q^2,p)(\C)$ with the Riemann surface $\Gamma_{0}(q^2)\cap\Gamma_{\mathrm{ns},1}^+(p)\backslash\uC^*$, where $\Gamma_{\mathrm{ns},1}^+(p)$ is $r_p^{-1}(C_1\cap\SL_2(\F_p))$ for some normaliser $C_1$ of a non-split Cartan subgroup of $\GL_2(\F_p)$ (recall that $r_p:\SL_2(\Z)\rightarrow\SL_2(\F_p)$ stands for the reduction modulo $p$). Similarly, we identify $X_{\rm{sp,ns}}^{-,+}(q,p)(\C)$ with the Riemann surface $\Gamma_{\rm{sp}}(q)\cap\Gamma_{\rm{ns}}^+(p)\backslash\uC^*$. Set $\Gamma:= \Gamma_{0}(q^2)\cap\Gamma_{\mathrm{ns},1}^+(p)$ and define 
\begin{equation*}
Q:=\begin{pmatrix} q & 0 \\ 0 & 1\end{pmatrix}.
\end{equation*}
The map $\Gamma\backslash\uC^*\rightarrow Q\Gamma Q^{-1}\backslash\uC^*$ given by $z\mapsto qz$ is an isomorphism. Note that $Q\Gamma Q^{-1}=\Gamma_{\rm{sp}}(q)\cap\Gamma_{\mathrm{ns},2}^+(p)$, where $\Gamma_{\mathrm{ns},2}^+(p)=r_p^{-1}(C_2\cap\SL_2(\F_p))$, where $C_2$ is a subgroup of $\GL_2(\F_p)$ conjugate to $C_1$. The Riemann surface $Q\Gamma Q^{-1}\backslash\uC^*$ corresponds to the $\C$-points of an algebraic curve $X_2$. The isomorphism between $X_{\rm{0,ns}}^{+}(q,p)(\C)$ and $X_2(\C)$ just defined can be seen to descend to an isomorphism defined over $\Q$. Therefore, we have a $\Q$-isomorphism between $X_{\rm{0,ns}}^{+}(q,p)$ and $X_2$. Now, we can define an isomorphism between $X_2$ and $X_{\rm{sp,ns}}^{-,+}(q,p)$ by a simple $\F_p$-base change. In a more formal way, we note that if $g\in\GL_2(\F_p)$ is such that $gC_2g^{-1}=C_1$, and if $X(p)$ denotes the modular curve parametrising elliptic curves with full $p$-torsion, then the automorphism of $X(p)$ defined by multiplication by $g$ induces a $\Q$-isomorphism $C_2\backslash X(p)\rightarrow C_1\backslash X(p)$. Moreover, this isomorphism preserves $j$-invariants. Since we have
\begin{equation*}
X_{\rm{sp,ns}}^{-,+}(q,p)=X_{\rm{sp}}(q)\times_{X(1)}C_1\backslash X(p)\quad\text{and}\quad X_2=X_{\rm{sp}}(q)\times_{X(1)} C_2\backslash X(p),
\end{equation*}
there is a $\Q$-isomorphism from $X_2$ to $X_{\rm{sp,ns}}^{-,+}(q,p)$. The isomorphism $\theta$ is obtained by composing this isomorphism with the isomorphism from $X_{0,\rm{ns}}^+(q,p)$ to $X_2$ defined above.

The statement relating the involutions of $X_{\rm{sp,ns}}^{-,+}(q,p)$ and $X_{0,\rm{ns}}^+(q,p)$ with the isomorphism $\theta$ can be achieved by looking at the moduli interpretation of $\theta$.

\end{proof}
\begin{rem}
The moduli intepretation of the isomorphism $\theta$ is given as follows. A $\C$-point in $X_{\rm{sp,ns}}^{-,+}(q,p)$ is represented by a tuple $(E,\varphi_1, \varphi_2, \frak{n})$, where $\varphi_1:E\rightarrow E_1$ and $\varphi_2:E\rightarrow E_2$ are two independent isogenies of degree $q$ and $\frak{n}$ is a necklace (for the definition of a necklace, see~\cite{rebwu}). The image of this point under $\theta$ is represented by the tuple $(E_1, \varphi_2\circ\hat{\varphi}_1, \varphi_1(\frak{n}))$, where $\hat{\varphi}_1$ stands for the dual isogeny of $\varphi_1$, and $\varphi_1(\frak{n})$ is the necklace in $E_1$ obtained as the image of the necklace $\frak{n}$ via $\varphi_1$.
\end{rem}
The curve $X_{0,\rm{ns}}^+(q^2,p)$ comes equipped with two ``degeneracy maps'' \begin{equation*}d_1,d_2:X_{0,\rm{ns}}^+(q^2,p)\rightarrow X_{0,\rm{ns}}^+(q,p)\end{equation*} coming from the degeneracy maps from $X_0(q^2)$ to $X_0(q)$. Let us briefly recall that the moduli interpretations of these degeneracy maps from $X_0(q^2)$ to $X_0(q)$ are as follows: a point in $X_0(q^2)$ represented by $(E,C)$ --- where $E$ is an elliptic curve and $C$ is a cyclic subgroup of $E(\C)$ of order $q^2$ --- is mapped by one of the degeneracy maps to $(E,C[q])$, and by the other to $(E/C[q], C/C[q])$. The maps $d_1$ and $d_2$ satisfy the relations
\begin{equation}\label{rel}
w_q\circ d_1=d_2\circ w_{q^2}\quad\text{and}\quad w_q\circ d_2=d_1\circ w_{q^2},
\end{equation}
where $w_q$ is the involution $X_{0,\rm{ns}}(q,p)$ coming from the Atkin--Lehner involution of $X_0(q)$. 

Let $J_{0,\rm{ns}}^+(q,p)$ stand for the Jacobian of $X_{0,\rm{ns}}^+(q,p)$. Adapting to our case a morphism from $X_0^+(q^2)$ to $J_0(q)$ that appears in section 3 of~\cite{maz_rat} and in~\cite{momose}, we define
\begin{equation*}
g:X_{0,\rm{ns}}^+(q^2,p)\rightarrow J_{0,\rm{ns}}^+(q,p)
\end{equation*}
by mapping a point $P$ to the class of $d_1(P)-d_2(P)$. By abuse of notation, we shall denote by $w_q$ the involution of $J_{0,\rm{ns}}^+(q,p)$ induced by the involution $w_q$ of $X_{0,\rm{ns}}^+(q,p)$. Equations~(\ref{rel}) give us the following equality:
\begin{equation}\label{rel2}
w_q\circ g=-g\circ w_{q^2}.
\end{equation}
Consider the abelian subvariety $B$ of $J_{0,\rm{ns}}^+(q,p)$ defined by $B:=(1+w_q)J_{0,\rm{ns}}^+(q,p)$. Define $J:=J_{0,\rm{ns}}^+(q,p)/B$ and let $\pi$ be the canonical projection from $J_{0,\rm{ns}}^+(q,p)$ to $J$. From equation~(\ref{rel2}) and Lemma~\ref{isom_mod}, we conclude that $\pi\circ g$ factors through $X_{\rm{sp,ns}}^{+,+}(q,p)$. Thus, we have the following commutative diagram:
\begin{equation*}
\begin{tikzpicture}[node distance = 2cm]

\node (A) {$X_{0,\rm{ns}}^+(q^2,p)$};
\node (B) [below of=A] {$X_{\rm{sp,ns}}^{+,+}(q,p)$};
\node (C) [node distance=4cm, right of=A]{$J_{0,\rm{ns}}^+(q,p)$};
\node (D) [below of=C] {$J$};
\draw[->] (A) to node [left] {} (B);
\draw[->] (A) to node [above] {$g$} (C);
\draw[->] (B) to node [below] {} (D);
\draw[->] (C) to node [right] {$\pi$} (D);

\end{tikzpicture}
\end{equation*}

The cusp at infinity $\infty$ of $X_{0,\rm{ns}}^+(q^2,p)$ is defined over $\Q(\zeta_p)^+:=\Q(\zeta_p+\zeta_p^{-1})$ (see Table~\ref{table:2}).
Note that $\pi\circ g(\infty)=0$.

\begin{prop}\label{fin_quot}
There exists a non-trivial optimal quotient $A$ of $J_{0,\rm{ns}}^+(q,p)$ such that $A(\Q)$ is finite and the kernel of the canonical projection $\pi':J_{0,\rm{ns}}^+(q,p)\rightarrow A$ is stable under the Hecke operators $T_{\ell}$, $\ell$ prime $\neq p$. Moreover, $\pi'$ factors through $\pi$.
\end{prop}
\begin{proof}
The first part of the proposition has been proved in~{\cite[Proposition~7.1]{darmer}} (even though this is only stated for the case where $q=2,3$, it is not hard to see that the same argument shows that the result holds for $q\in\{2,3,5,7,13\}$). In order to see that $\pi'$ factors through $\pi$, note that $A$ is defined to be the winding quotient of the new part of $J_{0,\rm{ns}}^+(q,p)$ (see~\cite{darmer}). Since $1+w_q$ is an element of the winding ideal, it follows from the definition of $J$ and $\pi$ that $\pi'$ factors through $\pi$.
\end{proof}

 Let $h$ denote the composition of $\pi\circ g$ with the natural projection from $J$ to $A$. Now, let $\loc$ be the ring of integers of $\Q(\zeta_p)^+$ and define $R:=\loc[1/2qp]$. Given a curve $X$ defined over $\Q$, we shall write $X_{/R}$ for the minimal regular model of $X$ over $R$. Similarly, given an abelian variety $B$, we shall write $B_{/R}$ for the N\'eron model of $B$ over $R$. With this notation, the morphism $h$ extends to a morphism $X_{0,\rm{ns}}^+(q^2,p)_{/R}\rightarrow A_{/R}$. By abuse of notation, we shall refer to this morphism by $h$ as well. 

Before stating our next result, let us recall the definition of formal immersion. Let $S_1$ and $S_2$ be two schemes and let $f:S_1\rightarrow S_2$ be a morphism. Let $x$ be a point in $S_1$ and define $y:=f(x)$. Write $\hat{\loc}_{S_1,x}$ and $\hat{\loc}_{S_2,y}$ for the formal completions of the local rings of $S_1$ and $S_2$ at $x$ and $y$, respectively. We say that $f$ is a formal immersion at $x$ if the induced morphism $\hat{f}_x:\hat{\loc}_{S_2,y}\rightarrow \hat{\loc}_{S_1,x}$ is surjective. 

Now, let $A$ be a Dedekind domain and suppose that $S_1$ and $S_2$ are schemes over $\Spec(A)$. Let $x$ be a section (over $A$) of $S_1$, and let $y$ be the section of $S_2$ which corresponds to the image of $x$. We will say that $f$ is a formal immersion at $x$ if $f$ is a formal immersion at $x_{\idp}$ for every non-zero prime ideal $\idp$ of $A$, where $x_{\idp}$ stands for the special fibre of $x$ at $\idp$.
\begin{prop}\label{standardformal}
The morphism $h$ is a formal immersion at $\infty_{/R}$, where $\infty_{/R}$ stands for the section over $R$ defined by $\infty$.
\end{prop}
\begin{proof}
The proof of this result is standard (see, for example,~\cite{maz_rat}). Indeed, let $\lambda$ be a prime of $K:=\Q(\zeta_p)^+$ not dividing~$2qp$. Let $\F_{\lambda}$ denote the residue field at $\lambda$ of $\Q(\zeta_p)^+$. Write $\Cot_{\infty}(X_{0,\rm{ns}}^{+}(q^2,p)_{/\F_{\lambda}})$ for the cotangent space of $X_{0,\rm{ns}}^{+}(q^2,p)_{/\F_{\lambda}}$ at $\infty_{/\F_{\lambda}}$. In a similar manner, write $\Cot(J_{0,\rm{ns}}^+(q,p)_{/\F_{\lambda}})$ for the cotangent space of $J_{0,\rm{ns}}^+(q,p)_{/\F_{\lambda}}$ at $0_{/\F_{\lambda}}$, and the same thing goes for $\Cot(A)$. Showing that $h$ is a formal immersion at $\infty_{/\F_{\lambda}}$ is equivalent to showing that the map $\Cot(A_{/\F_{\lambda}})\rightarrow\Cot_{\infty}(X_{0,\rm{ns}}^{+}(q^2,p)_{/\F_{\lambda}})$ is surjective.

 As the characteristic of $\lambda$ is different from $2$, $\Cot(A_{/\F_{\lambda}})$ injects into $\Cot(J_{0,\rm{ns}}^+(q,p)_{/\F_{\lambda}})$ (see {\cite[Corollary~1.1]{maz_rat}}). Since $A$ is non-trivial, there exists a non-trivial element $f\in\Cot(A_{/\F_{\lambda}})$. Regarding $f$ as an element of $\Cot(J_{0,\rm{ns}}^+(q,p)_{/\F_{\lambda}})$, let 
\begin{equation*}
f=\sum_{n=1}^{\infty} a_n(f) q^{n/p}\in\F_{\lambda}[[q^{1/p}]]
\end{equation*}
be the $q$-expansion of $f$. The image of $f$ in $\Cot_{\infty}(X_{0,\rm{ns}}^{+}(q^2,p)_{/\F_{\lambda}})$ is $a_1(f)$, as can be easily checked. If $a_1(f)\neq 0$ (in $\F_{\lambda}$), then we are done. Suppose, for the sake of contradiction, that $a_1(f)=0$ and  $a_1(T_{\ell}f)=0$ for every prime $\ell\neq p$. Now, $a_1(T_{\ell}f)=a_{\ell}(f)$, which yields that $a_n(f)=0$ for every $n$ coprime to $p$. Thus,
\begin{equation*}
f=\sum_{n=1}^{\infty}a_{pn}(f)q^n.
\end{equation*}
Therefore, $f$ is the reduction modulo $\lambda$ of a cusp form in $S_2(\Gamma_0(q))$. However, since $q\in\{2,3,5,7,13\}$, this vector space is trivial, which is a contradiction.
\end{proof}
\begin{corol}\label{corol_formim}
The morphism $X_{\rm{sp,ns}}^{+,+}(q,p)_{/R}\rightarrow A_{/R}$ is a formal immersion at $\infty_{/R}$.
\end{corol}

\begin{proof}[Proof of Proposition~\ref{int_j}]
Once again, the argument is standard. We start by noting that, given a $\Q$-rational point $P$ of $X_{\rm{sp,ns}}^{+,+}(q,p)$, its image $Q$ in $A$ is torsion, because the morphisms are defined over $\Q$ and $A$ has finite Mordell--Weil group. Let $\ell$ be a prime congruent to $\pm 1\bmod{p}$ (as $p\geq 11$, Proposition~\ref{goodp} asserts that these are the only primes we have to worry about). Since $p\geq 11$, we have $\ell>2$. Note that, since $\ell\equiv\pm 1\pmod{p}$, $\ell$ is inert in $\Q(\zeta_p)^+$. Let $\tilde{A}$ stand for the special fibre of the N\'eron model of $A$ over $\Z_{\ell}$. It is well-known that the reduction map gives us an injection $\Tors(A(\Q))\hookrightarrow \tilde{A}(\F_{\ell})$. Therefore, writing $\tilde{Q}$ for the reduction of $Q$ modulo $\ell$, we have $\tilde{Q}=0$ in $\tilde{A}(\F_{\ell})$ if, and only if, $Q=0$.

Suppose that $E$ has potentially multiplicative reduction at $\ell$. Then it gives rise to a $\Q$-rational point $P$ in $X_{\rm{sp,ns}}^{+,+}(q,p)$ which meets one of the cusps at the fibre at $\ell$. By choosing appropriate bases for $\GL_2(\F_q)$ and $\GL_2(\F_p)$, we may assume that this cusp is $\infty$. Therefore, writing, as above, $Q$ for the image of $P$ in $A$, we find that $\tilde{Q}=0$. Hence, by the observation of the previous paragraph, $Q=0$. Since the morphism $X_{\rm{sp,ns}}^{+,+}(q,p)_{/R}\rightarrow A_{/R}$ is a formal immersion at $\infty_{/R}$ in characteristic $\ell$, and as $P$ meets $\infty$ at the fibre of $\ell$, we must have $P=\infty$, which is a contradiction.
\end{proof}
We are finally ready to prove Theorem~\ref{qmaintheorem}. 
\begin{proof}[Proof of Theorem~\ref{qmaintheorem}]
We will make use of Proposition~\ref{int_j}. The argument used here is analogous to the one used in the proof of~{\cite[Theorem~1.1]{lem}}. Suppose that $E/\Q$ and $q$ are as in the statement of Theorem~\ref{qmaintheorem}. Due to Theorem~\ref{balak}, we can restrict ourselves to the case $q\in\{2,3,5,7\}$. Moreover, as the normalisers of non-split Cartan subgroups of $\GL_2(\F_2)$ are precisely its Borel subgroups, and as the case of Borel subgroups has already been treated by Theorem~\ref{lem1}, we can assume that $q\in\{3,5,7\}$. Suppose that there exists a prime $p>37$ for which $\bar{\rho}_{E,p}$ is not surjective. Then the image of $\bar{\rho}_{E,p}$ must be contained in the normaliser of a non-split Cartan subgroup of $\GL_2(\F_p)$. Proposition~\ref{int_j} now yields that the $j$-invariant of $E$ must be integral. Therefore, the elliptic curve $E$ gives rise to a $\Q$-rational point in $X_{\mathrm{sp}}^+(q)$ with integral $j$-invariant. By an appropriate choice of uniformisers, the $j$-invariant map $j:X_{\mathrm{sp}}^+(q)\rightarrow\proj^1$ can be explicitly described by one of the equations of Table~\ref{table:split} (the source of these equations is {\cite[p. 68]{chenthesis}}).
\begin{table}[h!]
\begin{tabular}{ |c|c| }
\hline
$q$ & $j$\\
\hline\hline
$3$ & $\frac{((t-9)(t+3))^3}{t^3}$\\
\hline
$5$ & $\frac{((t^2-5)(t^2+5t+10)(t+5))^3}{(t^2+5t+5)^5}$\\
\hline
$7$ & $\frac{((t^2-5t+8)(t^2-5t+1)(t^4-5t^3+8t^2-7t+7)(t+1))^3t}{(t^3-4t^2+3t+1)^7}$\\
\hline
\end{tabular}
\caption{Equations for the $j$-invariants of $X_{\mathrm{sp}}^+(q)$.}
\label{table:split}
\end{table}

Resorting to these equations, we are able to verify that there are only finitely many $\Q$-rational points in $X_{\mathrm{sp}}^+(q)$ with integral $j$-invariants. Moreover, the finitely many $j$-invariants associated to these points can be extracted from these equations: these are $-12288000, -884736, -32768, -5000, -1728, 0, 1728, 8000, 54000$ and $287496$. Of these, the only ones corresponding to elliptic curves without complex multiplication are $-5000$ and $-1728$. Thus, $j(E)\in\{-5000,-1728\}$. 

An example of an elliptic curve with $j$-invariant $-5000$ is the one given by the equation 
\begin{equation*}
E_1:y^2=x^3-x^2-208x+1412,
\end{equation*}
and an example of an elliptic curve with $j$-invariant $-1728$ is the one given by
\begin{equation*}
E_2:y^2=x^3-54x+216.
\end{equation*}
Upon consultation on the LMFDB database~\cite{lmfdb} --- where information about the image of mod $p$ Galois representations of elliptic curves was obtained using a method of Sutherland~\cite{suth} ---, we can observe that, if $p>37$, the representations $\bar{\rho}_{E_1,p}$ and $\bar{\rho}_{E_2,p}$ are both surjective. Recalling that any two elliptic curves without complex multiplication and sharing the same $j$-invariant are quadratic twists of each other, we conclude that $\bar{\rho}_{E,p}$ is surjective for every prime $p>37$, yielding a contradiction.
\end{proof}

It is worth highlighting that Theorem~\ref{balak} is only needed here to obtain an explicit upper bound for the non-surjective primes (which turns out to be $37$). If we were only interested in showing that there exists a constant $C$ such that $\bar{\rho}_{E,p}$ is surjective for every prime $p>C$ and every elliptic curve $E$ satisfying the conditions of Theorem~\ref{qmaintheorem}, then this could be achieved via Siegel's theorem as follows. Since the $j$-invariant map $j: X_{\rm{sp}}^+(13)\rightarrow\proj^1$ has more than two distinct points mapping to the point at infinity of $\proj^1$, Siegel's theorem asserts that there are only finitely many points in $X_{\rm{sp}}^+(13)(\Q)$ whose $j$-invariant is integral. Therefore, even without assuming that all the $\Q$-rational points of $X_{\rm{sp}}^+(13)$ are cuspidal or CM-points, we are still able to conclude that there are only finitely many isomorphism classes of elliptic curves satisfying the conditions of Theorem~\ref{qmaintheorem} and admitting a prime $p>37$ for which the Galois representation $\bar{\rho}_{E,p}$ is not surjective (recall that, under these conditions, the $j$-invariant of such an elliptic curve must be integral). We can now use Theorem~\ref{serthm} and the fact that, for elliptic curves without complex multiplication, the surjectivity of the Galois representation only depends on its isomorphism class to conclude the existence of our constant $C$.

\section{The case of $\Q$-curves}\label{therest}

We start by proving Theorem~\ref{noborel}. Let us just remark that if $K$ is a quadratic field and $E/K$ is an elliptic curve completely defined over $K$ and of degree $1$, then $E$ is defined over $\Q$. But theorems~\ref{borelmain}, \ref{cartanmain} and \ref{noborel} are already known to hold when $E$ is defined over $\Q$ (Theorem~\ref{cartanmain} for elliptic curves over $\Q$ is simply Theorem~\ref{qmaintheorem}, which we have just proved). Therefore, in everything that follows, whenever we speak of a $\Q$-curve, we will mean a $\Q$-curve that is \emph{not} defined over $\Q$. In the terminology of~\cite{lefourn}, these are known as \emph{strict} $\Q$-curves.
\subsection{Proof of Theorem~\ref{noborel}} 
In order to obtain Theorem~\ref{noborel} from the proposition above, we will use the following result of Le Fourn.
\begin{prop}[{\cite[Proposition~3.3]{lefourn}}]
Let $K$ be a quadratic field and let $E$ be a $\Q$-curve completely defined over $K$ and of square-free degree $d$. Assume, moreover, that the image of $\proj\bar{\rho}_{E,p}$ is contained in a Borel subgroup of $\PGL_2(\F_p)$ for some prime $p=11$ or $p\geq 17$ such that $p\nmid d$. Then $j(E)\in\loc_K$.
\end{prop}
The proof will be essentially an adaptation of an argument due to Mazur that can be found in sections 5, 6 and 7 of~\cite{maz_rat}.

From now to the end of this section, $K$ will be a quadratic number field, $E/K$ will be a $\Q$-curve completely defined over $K$, of square-free degree $d\geq 2$ and without complex multiplication, and $p\geq 13$ will be a prime number such that $p\nmid d$ and for which the image of $\proj\bar{\rho}_{E,p}$ is contained in a Borel subgroup of $\PGL_2(\F_p)$. We will assume that $p$ does not ramify in $K$. Consider the Galois representation $\bar{\rho}_{E,p}:G_K\rightarrow \GL_2(\F_p)$. As the image of $\proj\bar{\rho}_{E,p}$ is contained in a Borel subgroup of $\PGL_2(\F_p)$, we can choose a basis $P,Q$ of $E[p](\bar{K})$ such that $\langle P\rangle$ is a cyclic subgroup of $E(\bar{K})$ defined over $K$ (i.e., $\tau(\langle P\rangle)=\langle P\rangle$ for every $\tau\in G_K$). With respect to this basis, the representation $\bar{\rho}_{E,p}:G_K\rightarrow \GL_2(\F_p)$ has the shape
\begin{equation*}
\begin{pmatrix} \phi & *\\ 0 & \varphi\end{pmatrix},
\end{equation*}
where $\phi$ and $\varphi$ are two characters $G_K\rightarrow \F_p^{\times}$.
\begin{lem}
Let $\idp$ be a prime of $K$ dividing $p$. Then there exists a unique element $k\in\Z/(p-1)\Z$ and a character $\alpha:G_K\rightarrow \F_p^{\times}$ unramified at $\idp$ such that $\phi=\alpha\chi_p^k$, where $\chi_p$ stands for the mod $p$ cyclotomic character.
\end{lem}
\begin{proof}
Let $G_{\idp}$ be a decomposition subgroup of $G_K$ associated to $\idp$ and let $\loc_{\idp}$ denote the ring of integers of $K_{\idp}$, the completion of $K$ at $\idp$. The Artin map of class field theory gives us a continuous homomorphism $\loc_{\idp}^{\times}\rightarrow G_{\idp}^{\rm{ab}}$, from where we obtain another continuous map $\loc_{\idp}^{\times}\rightarrow\F_p^{\times}$ by composition with $\phi|_{G_{\idp}}$. Using the assumption that $p$ does not ramify in~$K$, it is easy to see that every continuous homomorphism $\loc_{\idp}^{\times}\rightarrow\F_p^{\times}$ must factor through $N:\loc_{\idp}^{\times}\rightarrow \Z_p^{\times}$, where $N$ stands for the norm map.  The result now follows from the fact that every continuous homomorphism $\Z_p^{\times}\rightarrow\F_p^{\times}$ is a power of the cyclotomic character (where we identify $\Z_p^{\times}$ with the inertia subgroup of $\Gal(\Q^{\rm{ab}}_p/\Q_p)$ via local class field theory).
\end{proof}
Let $\idp$ be a prime of $K$ lying above $p$. We now know that $\bar{\rho}_{E,p}$ has the shape
\begin{equation*}
\begin{pmatrix} \alpha\chi_p^k & *\\ 0 & \alpha^{-1}\chi_p^{1-k}\end{pmatrix},
\end{equation*}
where $\alpha$ is some character unramified at $\idp$. If $p$ remains prime in $K$, then, trivially, $\alpha$ is unramified at every prime of $K$ lying above $p$. The next lemma asserts that this is also true even if $p$ splits.

\begin{lem}
Using the above notation, $\alpha$ is unramified at every prime of $K$ lying above~$p$.
\end{lem} 
\begin{proof}
This is only true because we are assuming that the image of $\proj\bar{\rho}_{E,p}$ is contained in a Borel subgroup of $\PGL_2(\F_p)$. Let us start by recalling the notation introduced in Section~\ref{review}. For each element $\tau\in G_{\Q}$, we have a $K$-isogeny $\mu_{\tau}:{}^{\tau}E\rightarrow E$ satisfying the following conditions: if the restriction of $\tau$ to $K$ is the trivial automorphism of $K$, then $\mu_{\tau}$ is the identity; if, on the other hand, the restriction of $\tau$ to $K$ is the non-trivial automorphism of $K$, then $\mu_{\tau}$ has degree $d$ and, moreover, if $\tau'\in G_{\Q}$ is another element restricting to the non-trivial automorphism of $K$, then $\mu_{\tau}=\mu_{\tau'}$. Note that as the image of $\proj\bar{\rho}_{E,p}$ is contained in a Borel subgroup of $\PGL_2(\F_p)$, we have $\mu_{\tau}({}^{\tau}P)\in\langle P\rangle$ for every $\tau\in G_{\Q}$.

As the result trivially holds when $p$ remains prime in $K$, and as we are assuming that $p$ does not ramify in $K$, we will assume that $p$ splits in $K$. If this is the case, let $\frak{q}$ be the other prime of $K$ lying above $p$. Let $\sigma\in G_{\Q}$ be an element which restricts to the non-trivial automorphism of $K$. If $D_{\idp}$ is a decomposition subgroup of $G_K$ over $\idp$, then $D_{\frak{q}}:=\sigma D_{\idp}\sigma^{-1}$ is a decomposition subgroup of $G_K$ over $\frak{q}$. Moreover, if $I_{\idp}$ and $I_{\frak{q}}$ denote the corresponding inertia subgroups, we have $I_{	\frak{q}}=\sigma I_{\idp} \sigma^{-1}$. Therefore, every element of $I_{\frak{q}}$ can be uniquely written in the form $\sigma\tau\sigma^{-1}$ with $\tau\in I_{\idp}$. Let $\tau\in I_{\idp}$.  As any $\tau\in I_{\idp}$ acts as $\chi_p(\tau)^k$ on $\langle P\rangle$, and as $\mu_{\sigma}({}^{\sigma}P)\in\langle P\rangle$, we get  
\begin{equation*}
{}^{\tau\sigma^{-1}}P={}^{\tau}({}^{\sigma^{-1}}P)=\chi_p^k(\tau)({}^{\sigma^{-1}}P).
\end{equation*}
But then
\begin{equation*}
{}^{\sigma\tau\sigma^{-1}}P=\chi_p(\tau)^kP
\end{equation*}
for every $\tau\in I_{\idp}$. Therefore, the restriction of $\phi$ to $I_{\frak{q}}$ is $\chi_p^k$, proving that $\phi=\alpha\chi_p^k$ for some character $\alpha:G_K\rightarrow \F_p^{\times}$ unramified at every prime dividing $p$.
\end{proof}
\begin{lem}\label{kprops}
Using the above notation, there are integers $e\mid12$ and $a,b\in\{0,\ldots, e\}$ such that 

$(1)$ $e\leq 6$, 

$(2)$ $a+b=e$, 

$(3)$ $ek\equiv a\pmod{p-1}$ and 

$(4)$ $e(1-k)\equiv b\pmod{p-1}$.
\end{lem}
\begin{proof}
We know that $E$ has potentially good reduction at $\idp$. Therefore, after taking a field extension $L$ of $K_{\idp}$ with ramification degree dividing  $12$, but at most $6$, the curve $E$ acquires good reduction at $p$. Let $e$ denote the absolute ramification degree of $L$. As we are assuming that $p$ does not ramify in $K$, the integer $e$ is the ramification degree of $L$ over $K_{\idp}$. Let $I_{\idp}$ and $I_L$ denote the inertia subgroups of $G_{K_{\idp}}$ and $G_L$, respectively, and let $I_{\idp}^t$ and $I_L^t$ denote the respective tame inertia groups. Of course, $\phi$ and $\varphi$ factor through $I_{\idp}^t$. Let $\theta$ denote the fundamental character of level $1$ for $I_L$. We have $\chi_p=\theta^{e}$. Therefore, $\phi|_{I_L}=\theta^{ek}$ and $\varphi|_{I_L}=\theta^{e(1-k)}$. By a theorem of Raynaud, there are integers $a,b\in\{0,\ldots, e\}$ such that
\begin{equation*}
ek\equiv a\pmod{p-1}\quad\text{and}\quad e(1-k)\equiv b\pmod{p-1}.
\end{equation*}
In particular, we have $a+b\equiv e\pmod{p-1}$. However, $a+b\leq 2e\leq 12\leq p-1$, yielding \begin{equation*}a+b=e\leq 6,\end{equation*} as we wanted.
 
\end{proof}
Following the notation of Mazur~\cite{maz_rat}, we set $m:=(p-1)/2$, $n:=\num((p-1)/12)$ and $t:=m/n$.
\begin{lem}[cf.~{\cite[Lemma~5.3]{maz_rat}}]\label{unrever}
$\alpha^{2t}$ is unramified everywhere.
\end{lem}
\begin{proof}
The proof of this lemma is exactly the same as the one given by Mazur in~{\cite[Lemma~5.3]{maz_rat}}. For convenience of the reader, we reproduce it here. Let $S:=\Spec\Z[1/p]$. Consider the finite flat cyclic covering $X_1(p)_{/S}\rightarrow X_0(p)_{/S}$ of degree $(p-1)/2$. There is an intermediate cover
\begin{equation*}
X_1(p)_{/S}\rightarrow X_2(p)_{/S}\rightarrow X_0(p)_{/S}.
\end{equation*}
The only properties of the covering $X_2(p)_{/S}\rightarrow X_0(p)_{/S}$ that we are going to use are the following: it is a finite \'etale morphism of smooth $S$-schemes and its Galois group is isomorphic to the cyclic group $\Z/n\Z$, where $n:=\num((p-1)/12)$. This yields that the degree of $X_1(p)\rightarrow X_2(p)$ is $t$.

Our curve $E$ gives rise to a point $x=[(E,C_p)]\in X_0(p)(K)$. As all the coverings are cyclic, there exists a finite abelian extension $L/K$ for which there is a point $y=[(E',P')]\in X_1(p)(L)$ mapping to $x$. Moreover, as $X_2(p)_{/S}\rightarrow X_0(p)_{/S}$ is finite \'etale, the ramification degree of $L/K$ at any prime of characteristic different from $p$ divides $t$, and so it also divides~$6$. Now, as $y$ maps to $x$, there is an $L$-isomorphism $f:E\rightarrow E'$ mapping $C_p$ to $\langle P'\rangle$. The $L$-isomorphism $f$ is associated to an element of $H^1(\Gal(L/K),\Aut_{L}(E))$. However, $\Aut_L(E)=\{\pm 1\}$, as $E$ does not have complex multiplication. Therefore, given a prime $\lambda$ of $K$ of characteristic different from $p$, we find that $\alpha^t|_{I_{\lambda}}$ is a quadratic character, yielding that $\alpha^{2t}$ is unramified at $\lambda$. As we already know that $\alpha$ is unramified at any prime of characteristic $p$, we get the result.
\end{proof}

Let us just review what we have so far. The mod $p$ Galois representation $\bar{\rho}_{E,p}$ has the shape
\begin{equation*}
\begin{pmatrix} \alpha\chi_p^k & *\\ 0 & \alpha^{-1}\chi_p^{1-k}\end{pmatrix},
\end{equation*}
where $\chi_p$ is the mod $p$ cyclotomic character, $\alpha$ is a character such that $\alpha^{2t}$ is unramified everywhere, and $k$ satisfies the properties listed in Lemma~\ref{kprops}.

Let $\lambda$ be a prime of $K$ of characteristic different from $p$. Write $G_{\lambda}$ for a decomposition subgroup of $G_K$ over $\lambda$. Let $\alpha_{\lambda}$ denote the restriction of the character $\alpha$ to $G_{\lambda}$. As in section $6$ of~\cite{maz_rat}, we are going to split it in its ramified and unramified part. From local class field theory, we have a (non-unique) decomposition
\begin{equation*}
G_{\lambda}^{\rm{ab}}\cong \loc_{\lambda}^{\times}\times \hat{\Z},
\end{equation*}
where $\loc_{\lambda}$ denotes the ring of integers of $K_{\lambda}$. Sticking to the notation of Mazur, we write $\alpha_{\lambda}=\gamma_{\lambda}\cdot b_{\lambda}$, where the character $\gamma_{\lambda}$ factors through $\loc_{\lambda}^{\times}$ in the decomposition above and $b_{\lambda}$ is unramified. Lemma~\ref{unrever} implies that $\gamma_{\lambda}$ has order dividing $2t$. Let $L$ denote the splitting field of $\gamma_{\lambda}$. This is a totally ramified extension of $K_{\lambda}$ of degree dividing $2t$.
\begin{lem}\label{splitL}
Using the above notation, the elliptic curve $E$ has good reduction over $L$.
\end{lem}
\begin{proof}
Suppose, for the sake of contradiction, that $E$ does not have good reduction over $L$. Let $\F_q$ denote the residue field of $K$ ($q$ being its size) and $\tilde{E}/\F_q$ denote the special fibre of the N\'eron model of $E$ over $\loc_L$ (as $L$ is totally ramified, the residue field of $L$ is that of $K$). Note that we have
\begin{equation*}
\bar{\rho}_{E,p}|_{\Gal(\bar{L}/L)}\sim\begin{pmatrix} b_{\lambda}\chi_p^k & *\\ 0 & b_{\lambda}^{-1}\chi_p^{1-k}\end{pmatrix}.
\end{equation*}
Let $F$ be the splitting field of $b_{\lambda}\chi_p^k$. Then $F$ is an unramified extension of $L$ and $E(F)$ has a $p$-torsion point. Moreover, as N\'eron models are stable under \'etale base change, the special fibre of the N\'eron model of $E$ over $\loc_F$ is $\tilde{E}_{/\F_F}$, where $\loc_F$ is the ring of integers of $F$ and $\F_F$ is its residuel field. We conclude that there is a $p$-torsion point in $\tilde{E}(\F_F)$, which is clearly impossible when $E$ has bad reduction at $L$. Therefore, $E$ acquires good reduction at $L$.
\end{proof}
As a consequence, $\bar{\rho}_{E,p}|_{\Gal(\bar{L}/L)}$ factors through $\Gal(L^{\rm{unr}}/L)$. The Galois group $\Gal(L^{\rm{unr}}/L)$ is generated by the Frobenius automorphism $\Frob_{\lambda}$.
\begin{lem}\label{congruences}
Let $c$ denote the narrow class number of $K$. Then
\newline

$(1)$ $b_{\lambda}(\Frob_{\lambda})q^k+b_{\lambda}(\Frob_{\lambda})^{-1}q^{1-k}\equiv\Tr(\Frob_{\lambda})\pmod{p}$; and

$(2)$ $q^{12ck}+q^{12c(1-k)}\equiv\Tr(\Frob_{\lambda}^{12c})\pmod{p}$,
\newline

\noindent where $q$ is the size of the residue field of $K_{\lambda}$ and $\Tr(\Frob_{\lambda})\in \Z$ is the trace of the action of the Frobenius element of $\Gal(L^{\rm{unr}}/L)$ on the $p$-adic Tate module of $E$.
\end{lem}
\begin{proof}
Congruence $(1)$ follows from simply taking the trace of $\bar{\rho}_{E,p}(\Frob_{\lambda})$. Congruence $(2)$ follows from taking the trace of  $\bar{\rho}_{E,p}(\Frob_{\lambda}^{12c})$ and recalling (see Lemma~\ref{unrever}) that $\alpha^{12}$ is unramified everywhere (as $2t\mid 12$) and so, from class field theory, we have $\alpha^{12c}=1$.
\end{proof}
\begin{proof}[Proof of Theorem~\ref{noborel} for $p\equiv 1\pmod{4}$]
Let $\lambda$ be a prime of $K$ dividing $2$ and let $f$ denote the residual degree of $\lambda$. From Lemma~\ref{congruences}, we get
\begin{equation*}
q^{12ck}+q^{12c(1-k)}\equiv\Tr(\Frob_{\lambda}^{12c})\pmod{p},
\end{equation*}
where $q=2^f$. Using the notation and results of Lemma~\ref{kprops}, we have $e\mid 12$ and $e\leq 6$. Therefore, we can write $12=re$ for some integer $2\leq r\leq 12$. So,
 \begin{equation}\label{congruence}q^{rca}+q^{rcb}\equiv \Tr(\Frob_{\lambda}^{12c})\pmod{p},\end{equation} where $a,b$ are as in Lemma~\ref{kprops}. Now, by the Hasse--Weil bounds, \begin{equation*}|\Tr(\Frob_{\lambda}^{12c})|\leq 2\cdot q^{6c}.\end{equation*} 

We are now going to show that if $p\equiv1\pmod{4}$, then $2\cdot q^{6c}< q^{rca}+q^{rcb}$. Suppose, for the sake of contradiction, that we have $2\cdot q^6\geq q^{rca}+q^{rcb}$. If $ra>6$ (and so $rb=12-ra<6$), then it is easy to see that $2\cdot q^{6c}<q^{rca}+q^{rcb}$. By symmetry, we cannot have $ra<6$ either, nor $rb<6$, nor $rb>6$. Therefore, $ra=rb=6$, yielding one of the following cases:
\newline

(1) $r=2$, $e=6$ and $a=b=3$;

(2) $r=3$, $e=4$ and $a=b=2$; or

(3) $r=6$, $e=2$ and $a=b=1$.
\newline

Case (1) yields $6k\equiv 3\pmod{p-1}$, which is not possible, as $p$ is odd. For similar reasons, we cannot have case (3): here we would be forced to have $2k\equiv 1\pmod{p-1}$. We are only left with case (2). In this case, we obtain the congruence $4k\equiv 2\pmod{p-1}$. If $p\equiv 1\pmod{4}$, this is not possible. Thus, in this case, we must have
\begin{equation*}
2\cdot q^{6c}<q^{rca}+q^{rcb}.
\end{equation*}
  From this and from the congruence~(\ref{congruence}), we obtain a bound \begin{equation*}p\leq q^{rca}+q^{rcb}+2\cdot q^{6c}\leq 2\cdot q^{12c}+2\cdot q^{6c}=2^{6fc+1}(2^{6fc}+1),\end{equation*}as we wanted.\end{proof}

Let us now turn to the case where $p\equiv 3\pmod{4}$. As we have seen in the proof above, if we are not in any of the cases $(1)$, $(2)$ or $(3)$, then we obtain the bound $2^{6fc+1}(2^{6fc}+1)$. We therefore assume we are in one of these cases. Again, $(1)$ and $(3)$ cannot occur, so let us assume we are in case $(2)$. In other words, we are going to assume, from now on, that $r=3$, $e=4$ and $a=b=2$. As observed above, this yields $2k\equiv 1\pmod{m}$  (where, recall, $m$ was defined to be $(p-1)/2$), and, moreover, $t=1$ or $t=3$. Analogously to what is done in~\cite{maz_rat}, the aim of what follows is to show that every prime  $5\leq\ell<p/4$ unramified in $K$ and such that $\ell\nmid d$ satisfies
\begin{equation*}\left(\frac{\ell}{p}\right)=-1.\end{equation*}

After this has been proven, an application of Minkowski's bound for the norm of ideals in a class of the ideal class group will yield the theorem (cf. section~7 of~\cite{maz_rat}).

Before proceeding, let us make a remark that will be useful later on. Note that, as a consequence of $E$ not having complex multiplication, we have $\mu_{\sigma}\circ {}^{\sigma}\mu_{\sigma}=d$ or $-d$, where $\sigma\in G_{\Q}$ restricts to the non-trivial automorphism of $K$. 
\begin{lem}\label{quadtwist}
Let $K'$ be a quadratic extension of $K$. Let $E'$ be a $K'$-twist of $E$, and let $g:E_{/K'}\rightarrow E'_{/K'}$ be a $K'$-isomorphism. Then $\mu_{\sigma}':= g\circ \mu_{\sigma}\circ {}^{\sigma} g^{-1}$ is a $K$-isogeny from ${}^{\sigma}E'$ to $E'$ for every $\sigma\in G_{\Q}$. In particular, $E'$ is a $\Q$-curve completely defined over $K$ and of degree $d$. Moreover, $\mu'_{\sigma}\circ {}^{\sigma}\mu'_{\sigma}=\mu_{\sigma}\circ {}^{\sigma}\mu_{\sigma}$.
\end{lem}
\begin{proof}
All of these statements are easy to prove. If $E'$ is a trivial twist (i.e., if it is $K$-isomorphic to $E$), then the result is trivial. Suppose then that this is not the case. Consider the map $\tau\mapsto g^{-1}({}^{\tau}g)$, $\tau\in \Gal(K'/K)$. This is a $1$-cocycle $\Gal(K'/K)\rightarrow \Aut_{K'}(E_{K'})$. As $E$ does not have complex multiplication, $\Aut_{K'}(E_{K'})=\{\pm 1\}$, and so \begin{equation*}H^1(\Gal(K'/K),\Aut_{K'}(E_{K'}))=\Hom(\Gal(K'/K),\{\pm 1\}).\end{equation*} Thus, $\tau\mapsto g^{-1}({}^{\tau}g)$ is a quadratic character $\Gal(K'/K)\rightarrow \{\pm 1\}$. As we are assuming that $E'$ is not a trivial twist, we conclude that ${}^{\tau}g=-g$ if $\tau\in \Gal(K'/K)$ is the non-trivial element. Similarly, ${}^{\tau}({}^{\sigma}g)=-{}^{\sigma}g$ for every $\sigma\in G_{\Q}$. Therefore,
\begin{equation*}
{}^{\tau}\mu'_{\sigma}={}^{\tau}g\circ{}^{\tau}\mu_{\sigma}\circ {}^{\tau}({}^{\sigma}g^{-1})=\mu'_{\sigma},
\end{equation*}
meaning that $\mu'_{\sigma}$ is defined over $K$. It is clear that $\mu'_{\sigma}$ has degree $d$.

Finally,
\begin{equation*}
\mu_{\sigma}'\circ{}^{\sigma}\mu_{\sigma}'=g\circ\mu_{\sigma}\circ{}^{\sigma}g^{-1}\circ {}^{\sigma} g\circ {}^{\sigma}\mu_{\sigma}\circ g^{-1}=\mu_{\sigma}\circ {}^{\sigma}\mu_{\sigma},
\end{equation*}
as we wanted.
\end{proof}

As a consequence of this lemma, we may assume, after taking an appropriate quadratic twist if needed, that $E$ satisfies one of the following statements:
\newline

(A) $\mu_{\sigma}\circ {}^{\sigma}\mu_{\sigma}=d$ and $b_{\lambda}(\Frob_{\lambda})\neq -1$ for every prime $\lambda$ of $K$ of residual degree $2$ and of odd characteristic $<p/4$; 
\newline

(B) $\mu_{\sigma}\circ {}^{\sigma}\mu_{\sigma}=-d$ and $b_{\lambda}(\Frob_{\lambda})\neq 1$ for every prime $\lambda$ of $K$ of residual degree $2$ and of odd characteristic $<p/4$.
\newline

In order to treat the case where $p\equiv 3\pmod{4}$, we will resort to some general theory that can be consulted in~\cite{ribet_qcurv}.

Let $A/\Q$ be the abelian surface defined by $A:=\Res_{K/\Q}(E)$. This is a $\Q$-simple abelian variety of $\GL_2$-type. Let $F:=\Q\otimes\End_{\Q}(A)$. Then $F$ is either $\Q(\sqrt{d})$ or $\Q(\sqrt{-d})$, depending on whether $\mu_{\sigma}\circ {}^{\sigma}\mu_{\sigma}=d$ or $-d$, respectively (see section~7 of~\cite{ribet_qcurv}). Let $\frak{q}$ be a prime of $F$ over $p$. 
 If we denote by $\rho_{E,p}$ the Galois representation of $E$ obtained by the Galois action on the Tate module $V_p(E):=T_p(E)\otimes\Q_p$, then we have
\begin{equation*}
\rho_{A,\frak{q}}|_{G_K}\cong \rho_{E,p},
\end{equation*}
where $\rho_{A,\frak{q}}$ stands for the Galois representation obtained from the Galois action on $V_{\frak{q}}(A):=V_p(A)\otimes_{F\otimes \Q_p} F_{\frak{q}}$ (recall that $V_p(A)$ is  free of rank $2$ over $F\otimes \Q_p$). 
The reduction of $\rho_{A,\frak{q}}$ modulo $\frak{q}$ is well-defined up to semi-simplification, so we are going to denote by \begin{equation*}\bar{\rho}_{A,\frak{q}}:G_{\Q}\rightarrow\GL_2(\F_{\frak{q}})\end{equation*} this semi-simplified reduction. If we write $\bar{\rho}_{E,p}^{\rm{ss}}$ for the semi-simplification of $\bar{\rho}_{E,p}$, then $\bar{\rho}_{A,\frak{q}}|_{G_K}$ is isomorphic to $\bar{\rho}_{E,p}^{\rm{ss}}$. It can be easily verified that the condition that the image of $\proj\bar{\rho}_{E,p}$ is contained in a Borel subgroup of $\PGL_2(\F_p)$ implies that the image of $\bar{\rho}_{A,\frak{q}}$ is contained in a Borel subgroup of $\GL_2(\F_{\frak{q}})$, and so is contained in a split Cartan, as $\bar{\rho}_{A,\frak{q}}$ is semi-simple (see section 6 of~\cite{ribet_qcurv} and, in particular,~{\cite[Lemma~6.4]{ribet_qcurv}}).

Let us just mention a standard lemma that will be useful later. This is just a special case of~{\cite[Chapitre~III, 9.4, Proposition~6]{bourb}}, but it suffices for our purposes.
\begin{lem}\label{normendo}
Define $R_p:=F\otimes \Q_p$. Let $f\in\End_{R_p}(V_p(A))$. Let $P_f(T)\in R_p[T]$ be the characteristic polynomial of $f$. Regarding $f$ as an element of $\End_{\Q_p}(V_p(A))$, let $Q_f(T)\in\Q_p[T]$ be the characteristic polynomial of $f$. Let $\tau\in\Gal(F/\Q)$ be the non-trivial element. Then $\tau$ defines an automorphism of $R_p[T]$. Let $N_{F/\Q}:R_p[T]\rightarrow \Q_p[T]$ denote the map obtained by $h(T)\mapsto h(T){}^{\tau}h(T)$. Then \begin{equation*}
N_{F/\Q}(P_f(T))=Q_f(T).
\end{equation*}
\end{lem}

As $p$ is unramified in $K$, and as $\bar{\rho}_{A,\frak{q}}|_{G_K}$ is isomorphic to $\bar{\rho}_{E,p}^{\rm{ss}}$, we find that
\begin{equation*}
\bar{\rho}_{A,\frak{q}}\sim\begin{pmatrix} \beta\chi_p^k & 0\\0 & \theta\beta^{-1}\chi_p^{1-k}\end{pmatrix}
\end{equation*}
for some character $\beta:G_{\Q}\rightarrow\F_{\frak{q}}^{\times}$ unramified at $p$ such that $\beta|_{G_K}=\alpha$, and where $\theta:G_{\Q}\rightarrow\GL_2(\F_{\frak{q}})$ is a quadratic character defined as follows: if $\sigma\in G_{\Q}$ restricts to the non-trivial automorphism of $K$, then \begin{equation*} \theta(\sigma)=\frac{\mu_{\sigma}\circ {}^{\sigma}\mu_{\sigma}}{d};\end{equation*} otherwise, the image is $1$ (see section 7 of~\cite{ribet_qcurv}). In particular, if $\mu_{\sigma}\circ {}^{\sigma}\mu_{\sigma}=d$, then $F$ is real and $\theta=1$. 

\begin{notat}
In order to simplify exposition, from here on, given a rational prime $\ell$, we are going to assume we have fixed an embedding $\bar{\Q}\hookrightarrow \bar{\Q}_{\ell}$. This amounts to choosing a decomposition subgroup $G_{\ell}$ of $G_{\Q}$ over $\ell$. Moreover, every number field $L$ will be regarded as subfields of $\bar{\Q}$, so that, given a prime $\lambda$ of $L$ dividing $\ell$, we have an embedding of the decomposition subgroup $G_{\lambda}$ of $L$ over $\lambda$ into $G_{\ell}$. Similarly, algebraic extensions of $\Q_{\ell}$ will be regarded as subfields of $\bar{\Q}_{\ell}$.
\end{notat}

In what follows, $\lambda$ will be a prime of $K$ and $\ell$ will be the rational prime lying below $\lambda$. We will further assume that $5\leq \ell <p/4$, $\ell\nmid d$ and that $\ell$ does not ramify in $K$. Moreover, we will assume that $p$ is large enough so that it does not ramifiy in $F$. Write $\beta_{\ell}$ for the restriction of $\beta$ to $G_{\ell}$. As we did before, we can resort to class field theory to (non-uniquely) decompose $G_{\ell}^{\rm{ab}}$ as
\begin{equation*}
G_{\ell}^{\rm{ab}}\cong \Z_{\ell}^{\times}\times \hat{\Z},
\end{equation*}
and we obtain a decomposition $\beta_{\ell}=\eta_{\ell}\cdot \delta_{\ell}$, where $\eta_{\ell}$ factors through $\Z_{\ell}^{\times}$ and $\delta_{\ell}$ is unramified. Let $L'$ be the splitting field of $\eta_{\ell}$. It is a totally ramified extension of $\Q_{\ell}$. Moreover, if we keep writing $L$ for the splitting field of $\gamma_{\lambda}$ over $K_{\lambda}$ in the decomposition of $G_{\lambda}^{\rm{ab}}$, it can be easily checked that $L$ is an unramified extension of $L'$ of degree equal to that of $K_{\lambda}/\Q_{\ell}$. Therefore, the degree of $L'/\Q_{\ell}$ is the same as the degree of $L/K_{\lambda}$. In particular, it divides~$2t$. 
\begin{lem}\label{betadie}
Using the above notation, $\beta^{4t}=1$. Moreover, if $\ell$ splits in $K$, then $\beta_{\ell}^{2t}=1$.
\end{lem}
\begin{proof}
Recall that $\beta|_{G_K}=\alpha$, and that  $\alpha^{2t}$ is unramified at every prime of $K$ (see Lemma~\ref{unrever}). We claim that $\beta^{4t}$ is unramified everywhere. 

Let $\ell$ be a rational prime and let $\lambda$ a prime of $K$ dividing $\ell$. Let $I_{\ell}$ and $I_{\lambda}$ denote the inertia subgroups of $G_{\ell}$ and $G_{\lambda}$, respectively. Note that $[G_{\Q}:G_K]\leq 2$. Therefore, if $\tau\in I_{\ell}$, we have $\tau^2\in I_{\lambda}$. Therefore, $\beta(\tau)^{4t}=\beta(\tau^2)^{2t}=\alpha(\tau^2)^{2t}=1$, because $\alpha^{2t}$ is unramified everywhere. This shows that $\beta^{4t}$ is unramified everywhere.

As $\beta^{4t}$ is a character defined on $G_{\Q}$, it follows that it is trivial, which proves the first part of the lemma.

If $\ell$ is a rational prime splitting in $K$, then $G_{\ell}=G_{\lambda}$, and so we have $\beta(\tau)\in\F_p^{\times}$ for every $\tau\in G_{\ell}$. As $\beta(\tau)^{4t}=1$, we must have $\beta(\tau)^{2t}=\pm 1$. Note that $-1$ is not a quadratic residue modulo $p$, as $p\equiv 3\pmod{4}$. Therefore, we are forced to have $\beta(\tau)^{2t}=1$.
\end{proof}
\begin{lem}
Using the above notation, $A$ acquires good reduction over $L'$.
\end{lem}
\begin{proof}
Denoting the absolute Galois group of $L'$ by $G_{L'}$, what we have to show is that $\rho_{A,p}|_{G_{L'}}$ is unramified, where $\rho_{A,p}$ is the Galois representation of the Tate module $V_p(A)$ of~$A$. As $L$ is unramified over $L'$, it is enough to show that $\rho_{A,p}|_{G_{L}}$ is unramified. Since $A$ is the Weil restriction of $E$ from $K$ to $\Q$, and noting that $K\subseteq L$, we see that $\rho_{A,p}|_{G_{L}}$ is the direct sum of the Galois representations (restricted to $G_L$) associated to $V_p(E)$ and $V_p({}^{\sigma}E)$, where $\sigma\in \Gal(K/\Q)$ is non-trivial. As these two representations are isomorphic, it is therefore enough to show that $\rho_{E,p}|_{G_L}$ is unramified. But we already know from Lemma~\ref{splitL} that $E$ has good reduction over $L$, which yields that $\rho_{E,p}|_{G_L}$ is unramified, finishing the proof of the lemma.
\end{proof}
As a consequence, $\rho_{A,\frak{q}}|_{G_{L'}}$ factors through $\Gal((L')^{\rm{unr}}/L')$. Writing $\Frob_{\ell}$ for the Frobenius element of $\Gal((L')^{\rm{unr}}/L')$, we obtain a result analogous to Lemma~\ref{congruences}: 
\begin{equation}\label{finalcong}
\delta_{\ell}(\Frob_{\ell})\ell^k+\theta(\Frob_{\ell})\delta_{\ell}(\Frob_{\ell})^{-1}\ell^{1-k}\equiv a_{\ell}\pmod{\frak{q}},
\end{equation}
where $a_{\ell}\in \loc_F$ stands for the trace of $\rho_{A,\frak{q}}(\Frob_{\ell})$. If we denote by $P_{\ell}(T)$ the characteristic polynomial of $\rho_{A,\frak{q}}(\Frob_{\ell})$ (which has coefficients in $F$), then Lemma~\ref{normendo} asserts that $N_{F/\Q}(P_{\ell}(T))$ is precisely the characteristic polynomial of $\rho_{A,p}(\Frob_{\ell})$. As all the roots of the characteristic polynomial of $\rho_{A,p}(\Frob_{\ell})$ have complex size $\sqrt{\ell}$ (independently of the embedding into $\C$ chosen), we conclude that $|a_{\ell}|\leq 2\sqrt{\ell}$ for every embedding of $F$ into $\C$.

\begin{proof}[Proof of Theorem~\ref{noborel} for $p\equiv 3\pmod{4}$]
Suppose, for contradiction, that $\ell$ is a quadratic residue modulo $p$. Then $\ell^m\equiv 1\pmod{p}$. As $2k\equiv 1\pmod{m}$, we have $\ell^{k}\equiv\ell^{1-k}\pmod{p}$. We divide the proof in two parts: one to treat the cases where $\ell$ splits in $K$, and the other to treat the cases where $\ell$ remains prime.

Suppose that $\ell$ splits in $K$. Then $\theta(\Frob_{\ell})=1$. Equation~(\ref{finalcong}) yields
\begin{equation*}
\ell^{k}(\delta_{\ell}(\Frob_{\ell})+\delta_{\ell}(\Frob_{\ell})^{-1})\equiv a_{\ell}\pmod{\frak{q}}.
\end{equation*}
From Lemma~\ref{betadie}, and from the fact that $t=1$ or $t=3$, we conclude that either $\delta_{\ell}(\Frob_{\ell})$ is a $3$rd root of unity, or $-\delta_{\ell}(\Frob_{\ell})$ is. Thus, $\delta_{\ell}(\Frob_{\ell})+\delta_{\ell}(\Frob_{\ell})^{-1}=\pm 1$ or $\delta_{\ell}(\Frob_{\ell})+\delta_{\ell}(\Frob_{\ell})^{-1}=\pm 2$, and we find
\begin{equation*}
 \pm\ell^k\equiv a_{\ell}\pmod{\frak{q}}\quad\text{or}\quad  \pm 2\ell^k\equiv a_{\ell}\pmod{\frak{q}}.
\end{equation*}
Taking norms from $F$ to $\Q$, and recalling that $\ell^{2k}\equiv \ell\pmod{p}$, we get
\begin{equation*}
\ell\equiv N_{F/\Q}(a_{\ell})\pmod{p}\quad\text{or}\quad 4\ell\equiv N_{F/\Q}(a_{\ell})\pmod{p}.
\end{equation*}

As $|a_{\ell}|\leq 2\sqrt{\ell}$ for every embedding of $F$ in $\C$ and as $\ell <p/4$, we conclude that we must have $\ell=N_{F/\Q}(a_{\ell})$ or $4\ell =N_{F/\Q}(a_{\ell})$. In any case, $v_{\ell}(N_{F/\Q}(a_{\ell}))=1$, which is only possible if $\ell$ ramifies in $F$. However, $F=\Q(\sqrt{d})$ or $F=\Q(\sqrt{-d})$, and $\ell$ is an odd prime not dividing~$d$, so we obtain a contradiction. Therefore, if $3\leq \ell <p/4$, $\ell\nmid d$ and if $\ell$ splits in $K$, then $\ell$ is not a quadratic residue modulo $p$.

Suppose now that $\ell$ remains prime in $K$.  Assume that (A) holds. In this case, we have $\theta(\Frob_{\ell})=1$, because $\mu_{\sigma}\circ{}^{\sigma}\mu_{\sigma}=d$. We obtain the congruence
\begin{equation*}
\ell^k(\delta_{\ell}(\Frob_{\ell})+\delta_{\ell}(\Frob_{\ell})^{-1})\equiv a_{\ell}\pmod{\frak{q}}.
\end{equation*}
Also, from Lemma~\ref{betadie}, we know that $\delta_{\ell}^{4t}=1$, so either $\delta_{\ell}(\Frob_{\ell})^2$ is a $3$rd root of unity, or $-\delta_{\ell}(\Frob_{\ell})^2$ is. As, from (A), $b_{\lambda}(\Frob_{\lambda})\neq -1$, and as $b_{\lambda}(\Frob_{\lambda})=\delta_{\ell}(\Frob_{\ell})^2$, we see that $b_{\lambda}(\Frob_{\lambda})$ is either $1$, a primitive third root of unity or the negative of a primitive third root of unity. In the first case, we get $\delta_{\ell}(\Frob_{\ell})=\pm 1$, which leads to
\begin{equation*}
\pm2\ell^k\equiv a_{\ell}\pmod{\frak{q}}.
\end{equation*}
If, on the other hand, $b_{\lambda}(\Frob_{\lambda})$ is a primitive third root of unity, then, in particular, $\delta_{\ell}(\Frob_{\ell})^6=1$, which means that $\delta_{\ell}(\Frob_{\ell})^3$ is a square root of $1$. The situation where $\delta_{\ell}(\Frob_{\ell})=\pm 1$ takes us to the situation above, so we may assume that either $\delta_{\ell}(\Frob_{\ell})$ is a primitive third root of unity, or $-\delta_{\ell}(\Frob_{\ell})$ is. This leads to
\begin{equation*}
\pm\ell^k\equiv a_{\ell}\pmod{\frak{q}}.
\end{equation*} 
Finally, if $-b_{\lambda}(\Frob_{\lambda})$ is a primitive third root of unity, then $\delta_{\ell}(\Frob_{\ell})^6=-1$, which means that $\delta_{\ell}(\Frob_{\ell})\notin\F_p^{\times}$, as $p\equiv 3\pmod{4}$. In particular, $p$ remains prime in $F$. Moreover, as $\delta_{\ell}(\Frob_{\ell})^2\in\F_p^{\times}$, we see that the Galois conjugate of $\delta_{\ell}(\Frob_{\ell})$ is $-\delta_{\ell}(\Frob_{\ell})$. Thus, taking norms from $\F_{\frak{q}}$ to $\F_p$, we get
\begin{equation*}
-\ell^{2k}(\delta_{\ell}(\Frob_{\ell})^2+\delta_{\ell}(\Frob_{\ell})^{-2}+2)\equiv N_{F/\Q}(a_{\ell})\pmod{p},
\end{equation*}
and so
\begin{equation*}
-3\ell^{2k}\equiv N_{F/\Q}(a_{\ell})\pmod{p}.
\end{equation*}
Recalling that $\ell^{2k}\equiv \ell\pmod{p}$, and after taking norms from $F$ to $\Q$ in the appropriate cases, the three cases above give
\begin{equation*}
4\ell\equiv N_{F/\Q}(a_{\ell})\pmod{p}\quad\text{or}\quad \ell\equiv N_{F/\Q}(a_{\ell})\pmod{p}\quad\text{or}\quad -3\ell\equiv N_{F/\Q}(a_{\ell})\pmod{p}.
\end{equation*}
Using the same kind of arguments we used above, we conclude that $v_{\ell}(N_{F/\Q}(a_{\ell}))=1$, implying that $\ell$ ramifies in $F$ (recall that we are assuming that $\ell>3$), which it does not.

We omit the proof of the case where $\ell$ remains prime in $K$ and (B) holds, as it is treated in a similar manner to the case where (A) holds, except that now we have $\theta(\Frob_{\ell})=-1$.

We conclude that if $\ell$ is a prime satisfying $5\leq \ell <p/4$, $\ell\nmid d$ and if $\ell$ does not ramify in $K$, then $\ell$ is not a quadratic residue modulo $p$. In other words, $\ell$ remains prime in $\Q(\sqrt{-p})$. If $m_d$ is the number of prime divisors of $d$ and $m_K$ is the number of rational primes that ramify in $K$, then the number of primes $<p/4$ which do not remain prime in $\Q(\sqrt{-p})$ is $\leq m_d+m_K+2$. Therefore, there is an integer $M_{K,d}$ depending only of $K$ and $d$ such that the number of classes of the ideal class group of $\Q(\sqrt{-p})$ represented by an integral ideal of norm $<p/4$ is $\leq M_{K,d}$. However, a well-known result of Minkowski states that each class of the ideal class group is represented by an integral ideal of norm $<2\sqrt{p}/\pi$, which is a number smaller than $p/4$. This means that the class number of $\Q(\sqrt{-p})$ is bounded above by $M_{K,d}$. As there are only finitely many imaginary quadratic fields of a given class number, we conclude that $p$ can only be one of finitely many possibilities which only depend on $K$ and $d$. Theorem~\ref{noborel} follows.
\end{proof}

\subsection{The Borel case}

The aim of this section is to provide a proof of Proposition~\ref{borel_main} and Theorem~\ref{borelmain}. The arguments used to prove Proposition~\ref{borel_main} follow closely those of Ellenberg~\cite{ellen}.

Let $p$ and $q$ be as in the statement of Proposition~\ref{borel_main}. Define
\begin{equation*}
Z_{d,0}(q,p):=X_0(d)\times_{X(1)} X_{0,\rm{ns}}^+(q,p).
\end{equation*}
\begin{lem}\label{Kpt}
Let $w_d$ denote the involution of $Z_{d,0}(q,p)$ induced by the Atkin--Lehner involution of $X_0(d)$. Let $E$ be a $\Q$-curve as in the statement of Proposition~\ref{borel_main}. Then $E$ gives rise to a $K$-point $P$ in $Z_{d,0}(q,p)$ satisfying $w_dP={}^{\sigma}P$ for every $\sigma\in G_{\Q}$ restricting to the non-trivial element of $\Gal(K/\Q)$.
\end{lem}
\begin{proof}
The proof of this result is identical to the proof of~{\cite[Proposition~2.2]{ellen}}.
\end{proof}
 As in \cite{ellen}, we are going to consider a suitable quadratic twist of $Z_{d,0}(q,p)$ whose $\Q$-rational points will correspond to $\Q$-curves completely defined over $K$, of degree~$d$, without complex multiplication and with level structures at $q$ and $p$ corresponding to the curve $X_{0,\rm{ns}}^+(q,p)$ (i.e., $\Q$-curves satisfying the conditions of Proposition~\ref{borel_main}).

Define the homomorphism $\psi:\Gal(K/\Q)\rightarrow\Aut_{\Q}(Z_{d,0}(q,p))$ by mapping $\sigma$, the non-trivial element of $\Gal(K/\Q)$, to $w_d$, the involution of $Z_{d,0}(q,p)$ induced by the Atkin--Lehner operator associated to $X_0(d)$. Let $Z_{d,0}^{\psi}(q,p)$ be a quadratic twist associated to $\psi$. By definition, $Z_{d,0}^{\psi}(q,p)$ is a curve defined over $\Q$ for which there exists a $K$-isomorphism $\varphi:Z_{d,0}(q,p)_{/K}\rightarrow Z_{d,0}^{\psi}(q,p)_{/K}$ such that $\varphi\circ w_d={}^{\sigma}\varphi$. This isomorphism yields a bijection between the sets $Z_{d,0}^{\psi}(q,p)(\Q)$ and $\{P\in Z_{d,0}(q,p)(K): w_dP={}^{\sigma} P\}$. By Lemma~\ref{Kpt}, we conclude that a $\Q$-curve as in Proposition~\ref{borel_main} gives rise to a $\Q$-point in $Z_{d,0}^{\psi}(q,p)$.

There is a natural degeneracy map $\delta:Z_{d,0}(q,p)\rightarrow X_{0,\rm{ns}}^+(q,p)$. Let $f:X_{0,\rm{ns}}^+(q,p)\rightarrow A$ stand for the morphism in~{\cite[Lemma~8.2]{darmer}} (where, as in~\cite{darmer}, $A$ is the winding quotient of the Jacobian of $X_{0,\rm{ns}}^+(q,p)$).  Note that this map is only defined over $\Q(\zeta_p)^+$. We define two morphisms $\gamma_1,\gamma_2:Z_{d,0}(q,p)\rightarrow A$ by 
\begin{equation*}
\gamma_1:= f\circ\delta\quad\text{and}\quad \gamma_2:=f\circ\delta\circ w_d.
\end{equation*}
Moreover, we define
\begin{equation*}
h_1:=\gamma_1\circ\varphi^{-1}\quad\text{and}\quad h_2:\gamma_2\circ\varphi^{-1}.
\end{equation*}
These two maps are defined over $L:=K(\zeta_p+\zeta_p^{-1})$. We finally set $h:=h_1+h_2$.

Denote by $\loc_L$ the ring of integers of $L$ and define $R:=\loc_L[1/6qp]$. Using the notation of section~\ref{rational}, $h$ can be extended to a morphism $Z_{d,0}^{\psi}(q,p)_{/R}\rightarrow A_{/R}$. By abuse of notation, we shall denote this morphism by $h$ as well.

In what follows, the point at infinity of $Z_{d,0}^{\psi}(q,p)$ is, of course, defined to be the image of the point at infinity of $Z_{d,0}(q,p)$ via $\varphi$.
\begin{lem}\label{form_im1}
The morphism $h$ is a formal immersion at $\infty_{/R}$.
\end{lem}
\begin{proof}
The arguments of the proof are essentially the ones used in the proof of~{\cite[Proposition~3.2]{ellen}}. For the convenience of the reader, we will present the proof here. We first note that it is enough to show that the morphism $\gamma:=\gamma_1+\gamma_2$ is a formal immersion at $\infty_{/R}$. Let $\lambda$ be a prime ideal of $R$ and let $\F_{\lambda}$ be the associated residue field. Writing $\Cot(A_{/\F_{\lambda}})$ for the cotangent space of $A_{/\F_{\lambda}}$ at $0$, and $\Cot_{\infty}(Z_{d,0}(q,p)_{/\F_{\lambda}})$ for the cotangent space of $Z_{d,0}(q,p)_{/\F_{\lambda}}$ at $\infty_{/\F_{\lambda}}$, it is enough to show that the map
\begin{equation*}
\gamma_{/\F_{\lambda}}^*:\Cot(A_{/\F_{\lambda}})\rightarrow \Cot_{\infty}(Z_{d,0}(q,p)_{/\F_{\lambda}})
\end{equation*}
induced by $\gamma$ is surjective.

Recall that, by definition, $\gamma_{1/\F_{\lambda}}$ factors as
\begin{equation*}
Z_{d,0}(q,p)_{/\F_{\lambda}}\xrightarrow{\delta_{/\F_{\lambda}}} X_{0,\rm{ns}}^+(q,p)_{/\F_{\lambda}}\xrightarrow{f_{/\F_{\lambda}}} A_{/\F_{\lambda}},
\end{equation*}
while $\gamma_{2/\F_{\lambda}}$ factors as
\begin{equation*}
Z_{d,0}(q,p)_{/\F_{\lambda}}\xrightarrow{(\delta\circ w_d)_{/\F_{\lambda}}}X_{0,\rm{ns}}^+(q,p)_{/\F_{\lambda}}\xrightarrow{f_{/\F_{\lambda}}} A_{/\F_{\lambda}}.
\end{equation*}
Since $(\delta\circ w_d)_{/\F_{\lambda}}$ is ramified at $\infty_{/\F_{\lambda}}$, we conclude that the map $\gamma_{2/\F_{\lambda}}^*$ induced on the cotangent spaces is $0$. Hence,
\begin{equation*}
\gamma_{/\F_{\lambda}}^*=\gamma_{1/\F_{\lambda}}^*.
\end{equation*}
On the other hand, $\delta_{/\F_{\lambda}}$ is unramified at $\infty_{/\F_{\lambda}}$. Moreover, the morphism $f:X_{0,\rm{ns}}^+(q,p)\rightarrow A$ has been proven to be a formal immersion at infinity in~{\cite[Lemma~8.2]{darmer}} (once again, we remark that the arguments used in~\cite{darmer} hold when $q\in\{2,3,5,7,13\}$, even though this result is only stated for $q\in\{2,3\}$). It follows that $\gamma^*_{1/\F_{\lambda}}$ surjects onto $\Cot_{\infty}(Z_{d,0}(q,p)_{/\F_{\lambda}})$, and, consequently,  so does $\gamma^*_{/\F_{\lambda}}$.
\end{proof}
\begin{lem}\label{torsion}
Let $P$ be a $\Q$-rational point in $Z_{d,0}^{\psi}(q,p)$. Then $h(P)$ is a torsion point of~$A(\bar{\Q})$.
\end{lem}
\begin{proof}
This argument was used in the proof of~{\cite[Lemma~8.3]{darmer}}. Let $\tau\in G_{\Q}$. It is easy to check that ${}^{\tau} (h(P))-h(P)$ is a cuspidal divisor. Therefore, by the theorem of Manin--Drinfeld, there exists a positive integer $m$ such that \begin{equation*}m({}^{\tau} (h(P)))=mh(P).\end{equation*} This means that $mh(P)$ is defined over $\Q$. As $A(\Q)$ is finite, we conclude that $mh(P)$ is torsion, and so is $h(P)$.
\end{proof}
\begin{proof}[Proof of Proposition~\ref{borel_main}]
Note that if $q\geq 11$ is a prime different from $13$, then~{\cite[Proposition~3.2]{ellen}} yields that $E$ has potentially good reduction at every prime of characteristic $>3$. As the image of $\proj\bar{\rho}_{E,p}$ is contained in the normaliser of a non-split Cartan subgroup of $\PGL_2(\F_p)$, Proposition~\ref{goodp} yields that $E$ cannot have potentially multiplicative reduction at primes above $2$ and $3$ either. Therefore, $E$ has potentially good reduction everywhere. In other words, the $j$-invariant of $E$ lies in $\loc_K$. 

We are reduced to proving the cases where $q\in\{2,3,5,7, 13\}$. As noted above, a $\Q$-curve as in the statement of the Proposition~\ref{borel_main} gives rise to a $\Q$-rational point $P$ in $Z_{d,0}^{\psi}(q,p)$. If $\lambda$ is a non-archimedean prime of $K$ such that $N_{K/\Q}(\lambda)^2\not\equiv 1\pmod{p}$, then Proposition~\ref{goodp} asserts that $E$ has potentially good reduction at $\lambda$, as $p\geq 11$. Suppose that $N_{K/\Q}(\lambda)^2\equiv 1\pmod{p}$. Note that under this condition $\lambda$ remains a prime in $L$. As $p\geq 11$, the prime $\lambda$ does not divide $6$. Suppose, for the sake of contradiction, that $E$ has potentially multiplicative reduction at $\lambda$. Then the section of $Z_{/R}$ corresponding to $P$ meets a cusp in the fibre above $\lambda$. By changing bases if needed, we may assume that this cusp is $\infty$. Now, we know that $h(P)\in\Tors(A(L))$. Moreover, the torsion subgroup of $A(L)$ injects, via reduction modulo $\lambda$, into $A(\F_{\lambda})$. But $h(P)$ meets $h(\infty)$ in the special fibre above $\lambda$. Therefore, $h(P)=h(\infty)$. Since $h$ is a formal immersion at $\infty_{/\F_{\lambda}}$, we conclude that $P=\infty$, which is absurd.
\end{proof}
\begin{proof}[Proof of Theorem~\ref{borelmain}]
By Theorem~\ref{noborel}, the prime $q$ belongs to a finite list. For each one of these primes, we obtain a $K$-rational point in $X_0(d)\times_{X(1)} X_0(q)$, which is a modular curve with at least three cusps. By an argument due to Serre (see~{\cite[Lemme~18]{ser_que}}), we know that there exists a constant $C'_K$, depending only on the number field $K$, such that, for $p>C_K'$, the image of $\proj\bar{\rho}_{E,p}$ is not exceptional. By Theorem~\ref{noborel}, there is another constant $C_{K,d}''$ such that, for $p>C_{K,d}''$, it is also not contained in a Borel subgroup. Therefore, if $p>C_{K,d}''$, the image of $\proj\bar{\rho}_{E,p}$ is contained in the normaliser of a Cartan subgroup (split or non-split). If it is contained in the normaliser of a split Cartan subgroup, then we can use the result of Le Fourn that we stated as Proposittion~\ref{cartan_main} to conclude that $j(E)\in \loc_K$. In the case where the image of $\proj\bar{\rho}_{E,p}$ is contained in the normaliser of a non-split Cartan, we use Proposition~\ref{borel_main} that we have just proven in order to, once again, conclude that $j(E)\in\loc_K$. In any case, there is a constant $C''_{K,d}$ such that, if $p$ is a prime $>C''_{K,d}$, then either $\proj\bar{\rho}_{E,p}$ is surjective, or $j(E)$ is integral. As the modular curve $X_0(d)\times_{X(1)}X_0(q)$ has at least three cusps, Siegel's theorem asserts that there are only finitely many points in $X_0(d)\times_{X(1)}X_0(q)(K)$ whose $j$-invariants are in $\loc_K$. As $q$ is in a finite list of primes, we obtain a finite list of $j$-invariants of $\Q$-curves satisfying the conditions of the theorem and for which there exists a prime $p>C_{K,d}''$ with $\bar{\rho}_{E,p}$ non-surjective. Noting that surjectiveness only depends on the $j$-invariant if $j(E)\neq 0,1728$ (as any two elliptic curves with the same $j$-invariant are quadratic twists of each other as long as $j\neq 0,1728$), we can now use the theorem of Serre that we presented in the introduction as Theorem~\ref{serthm} for each one of these finitely many $j$-invariants, and we obtain the result.
\end{proof}
\subsection{The normaliser of a split Cartan case}
The proof of Theorem~\ref{cartanmain} is a simple exercise using Le Fourn's Proposition~\ref{cartan_main}.
\begin{proof}[Proof of Theorem~\ref{cartanmain}]
We may assume that $d\geq 2$, as the case $d=1$ is precisely the one treated by Theorem~\ref{qmaintheorem}. The aim is to show that, for each $d$ as in the statement of the theorem, there are only finitely many points in $X_0(d)(K)$ with $j$-invariant in $\loc_K$. The result follows, as, by the same arguments employed in the proof of Theorem~\ref{borelmain}, there exists a constant $C_{K,d}''$ (keeping up with the notation used in the proof of Theorem~\ref{borelmain}) such that, if $p$ is a prime $>C_{K,d}''$ and $E/K$ is a curve satisfying the conditions of the theorem, then either $\proj\bar{\rho}_{E,p}$ is surjective, or $j(E)\in \loc_K$, and then we only have to use Serre's Theorem~\ref{serthm} for each of the finitely many points of $X_0(d)(K)$ in the same way we used it in the proof of Theorem~\ref{borelmain}.

If the genus of $X_0(d)$ is $\geq 2$, then a theorem of Faltings asserts that the set $X_0(d)(K)$ is finite, and we are done. If the genus of $X_0(d)$ is $1$, then the finiteness of the number of points in $X_0(d)(K)$ with integral $j$-invariant comes from a theorem of Siegel. Finally, if the genus $X_0(d)$ is $0$, then, as we require $d$ to be square-free and not in the set $\{2,3,5,7,13\}$, we conclude that $d$ is a product of two distinct primes. Therefore, $X_0(d)$ has at least three cusps, and we can use Siegel's theorem in order to conclude the finiteness of the number of points in $X_0(d)(K)$ with integral $j$-invariant.
\end{proof}

\bibliography{final_serreproblem}
\bibliographystyle{abbrv}

\end{document}